\documentclass{article}

\usepackage{parskip}

\usepackage{layout} 
\usepackage[top=3cm, bottom=3cm, left=2cm, right=2cm]{geometry} 
\usepackage[british]{babel}
\usepackage[utf8]{inputenc}
\usepackage[T1]{fontenc}
\usepackage{amsmath}
\usepackage{amssymb}
\usepackage{mathrsfs}
\usepackage{amsthm}
\usepackage{enumerate} 
\usepackage{color, colortbl} 
\usepackage{soul} 
\usepackage{fancyhdr, fancybox, fancyvrb,fancyref} 
\usepackage{tcolorbox}
\usepackage[compact]{titlesec}
\usepackage{graphicx} 
\usepackage[tikz]{bclogo}
\usepackage{enumitem}
\usepackage{comment}
\usepackage{mathtools}
\usepackage{esint}
\usepackage{faktor}
\usepackage[colorlinks=true,urlcolor=blue,linkcolor=blue]{hyperref}
\usepackage{csquotes}
\usepackage{aliascnt}

\newtheorem{theoremIntro}{Theorem}

\newaliascnt{propositionIntro}{theoremIntro}
\newtheorem{propositionIntro}[propositionIntro]{Proposition}
\aliascntresetthe{propositionIntro}

\newtheorem{theorem}{Theorem}[section]

\newaliascnt{proposition}{theorem}
\newtheorem{proposition}[proposition]{Proposition}
\aliascntresetthe{proposition}

\newaliascnt{corollary}{theorem}
\newtheorem{corollary}[corollary]{Corollary}
\aliascntresetthe{corollary}

\newaliascnt{lemma}{theorem}
\newtheorem{lemma}[lemma]{Lemma}
\aliascntresetthe{lemma}

\newaliascnt{remark}{theorem}
\newtheorem{remark}[remark]{Remark}
\aliascntresetthe{remark}

\newaliascnt{defi}{theorem}

\aliascntresetthe{defi}

\newaliascnt{claim}{theorem}

\aliascntresetthe{claim}

\usepackage[capitalize]{cleveref}

\newcommand{\scal}[2]{\left\langle #1,#2 \right\rangle}
\newcommand{\g}{\nabla}

\newcommand{\di}{\mathrm{div}}
\newcommand{\lap}{\Delta}
\newcommand{\dr}{\partial}
\newcommand{\vol}{\mathrm{vol}}

\newcommand{\Span}{\mathrm{Span}}

\newcommand{\tr}{\mathrm{tr}}

\newcommand{\Riem}{\mathrm{Riem}}
\newcommand{\Ric}{\mathrm{Ric}}
\newcommand{\Scal}{\mathrm{Scal}}

\newcommand{\Conf}{\mathrm{Conf}}

\newcommand{\sign}{\mathrm{sign}}
\newcommand{\Ima}{\mathrm{Im}}

\newcommand{\R}{\mathbb{R}}

\newcommand{\N}{\mathbb{N}}

\newcommand{\s}{\mathbb{S}}

\newcommand{\RP}{\mathbb{RP}}

\newcommand{\B}{\mathbb{B}}

\newcommand{\Er}{\mathcal{E}}

\newcommand{\Sr}{\mathcal{S}}
\newcommand{\pr}{\mathcal{P}}
\newcommand{\Arond}{\mathring{A}}

\newcommand{\inter}[2]{[\![#1,#2]\!]}

\newcommand{\ve}{\varepsilon}

\newcommand{\ggo}{\mathfrak{g}}
\newcommand{\dd}{\mathrm{d}}

\newcommand{\Egr}{\mathcal{E}_{\mathrm{GR}}}
\newcommand{\Ep}{\mathcal{E}_{\mathrm{P}}}

\numberwithin{equation}{section}

\title{A duality theorem for a four-dimensional Willmore energy}
\author{Dorian Martino\footnote{Partial support through ANR BLADE-JC ANR-18-CE40-002 is acknowledged.\\Institut de Mathématiques de Jussieu, Université Paris Cité, Bâtiment
		Sophie Germain, 75205 Paris Cedex 13, France.\\ETH Zürich, 101 Rämistrasse, 8092 Zürich.\\ Email: dorian.martino@math.ethz.ch}}
\date{}

\begin{document}
	
	\maketitle

	\begin{abstract}
		We prove an analogue of the energy identities underlying Bryant's duality theorem for a four-dimensional Willmore energy $\Egr$ obtained by Graham--Reichert and Zhang in codimension one. We show that, for an immersion $\Phi$ of a compact four-dimensional manifold without boundary $\Sigma$ into $\R^5$, the energy $\Egr(\Phi)$ is equal to two energies associated with its conformal Gauss map $Y$: one defined only in terms of the image of $Y$, which is the analogue of the area functional for Willmore surfaces, and another defined on maps from $\Sigma$ into the de Sitter space $\s^{5,1}$, which is the analogue of the Dirichlet energy for Willmore surfaces. We prove that, even when restricted to immersions of a given topological manifold $\Sigma^4$, $\Egr$ is not bounded below on the set of immersions from $\Sigma$ into $\R^5$. We exhibit a second conformally invariant energy $\Ep$ that is bounded below and whose construction is closer to that of the two-dimensional Willmore energy.
	\end{abstract}

	\section{Introduction}
	
	A standard problem in conformal geometry is to characterise the functionals defined on submanifolds of a given Riemannian manifold that are invariant under conformal changes of the ambient metric. For instance, for immersions $\Phi$ of a closed surface $\Sigma$ into $\R^3$, Mondino--Nguyen \cite{mondino2018} proved that any conformally invariant functional that is polynomial in the second fundamental form can be expressed as a linear combination:
	\begin{align*}
		\int_\Sigma \left(\alpha |\Arond|^2_g +\beta K\right)\, \dd\vol_g,
	\end{align*}
	where $\alpha,\beta\in\R$ are fixed constants, $g$ is the first fundamental form of $\Phi$, $\Arond$ is the traceless part of the second fundamental form $A$ of $\Phi$, and $K$ is its Gauss curvature. By Gauss--Bonnet, the term $\int_\Sigma K \dd\vol_g=2\pi\chi(\Sigma)$ depends only on the topology of $\Sigma$, not on its immersion into $\R^3$. Up to a constant depending only on the topology of $\Sigma$, the functional $\int_\Sigma |\Arond|^2_g\, \dd\vol_g$ can be expressed in terms of $\int_\Sigma H^2\, \dd\vol_g$, which is the Willmore functional and has been studied since the nineteenth century, see for instance \cite{bryant1984,helein1998,hertrich-jeromin2003,kuwert2001,kuwert2004,kuwert2012,laurain,marque2021,marques2014b,martino2023,michelat,pinkall1985,riviere2016a,riviere2021}. 
	
	In higher dimensions, the diversity of curvature quantities makes the study of conformally invariant energies defined on immersions $\Sigma^n \to \R^{n+1}$ of a given closed $n$-manifold $\Sigma$ much more complex. For instance, one can consider functionals containing only polynomial contractions of intrinsic curvature tensors or, more generally, functionals defined by polynomials in the second fundamental form. Mondino--Nguyen \cite{mondino2018} proved that the first family can be expressed as linear combinations of the Pfaffian and scalar contractions of the Weyl tensor, whereas the second family can be expressed solely in terms of the traceless part of the second fundamental form. For instance, in dimension $n=4$, any linear combination of the following functionals is conformally invariant:
	\begin{align*}
		\int_\Sigma |\Arond|^4_g\, \dd\vol_g, \qquad \int_\Sigma \tr_g(\Arond^4)\, \dd\vol_g,\qquad \int_\Sigma \det_g\Arond\, \dd\vol_g.
	\end{align*}
	A central problem is then to construct and classify conformally invariant functionals containing derivatives of the second fundamental form. This question has been extensively studied, and examples and methods for constructing such functionals can be found in \cite{blitz2023generalized,gover2020,graham1999,guven,vyatkin,zhang}. By adding polynomial terms in the second fundamental form, one can produce conformally invariant functionals with a leading term of the form
	\begin{align*}
		\int_\Sigma |\g A|^2_g\, \dd\vol_g, \qquad  \int_\Sigma |\g \Arond|^2_g\, \dd\vol_g,\qquad \int_\Sigma |\g H|^2_g\, \dd\vol_g.
	\end{align*}
	It turns out that the three functionals above differ only by lower-order terms, see \Cref{lm:integral_gH} in the appendix. Blitz--Gover--Waldron have listed some of these functionals and compared them in \cite[Section 4]{blitz2023generalized}. The study of critical points of conformally invariant energies can also be stated in terms of a singular Yamabe problem, see for instance \cite{gover2021,graham2017,juhl}. 
	
	Recently, the search for a suitable generalisation of the Willmore energy for immersions of compact manifolds $\Sigma^4$ into $\R^5$ has been the subject of much work \cite{gover2020,guo,robingraham2020}. One of the main motivations is the study of volume renormalisation in the context of the AdS/CFT correspondence in physics. Broadly speaking, this correspondence asserts the existence of a duality between certain gravitational theories on an asymptotically anti-de Sitter spacetime $(X^{d+2},g)$ and certain conformal field theories on the conformal boundary $\dr_\infty X$ of $X$, see for instance \cite{maldacena1997,aharony1999}. In \cite{maldacena1998,graham1999,ryu2006,perlmutter2015}, the authors argue that computing quantities such as entanglement entropy or the expectation value of Wilson lines requires evaluating the total area of a minimal hypersurface $M^{d+1}$ in $X^{d+2}$. However, the area of $M$ need not be finite because the boundary of $X$, and thus the boundary of $M$, may lie at infinity. To obtain rigorously defined quantities, the renormalisation procedure detailed in \cite{fefferman2012,graham1999,gover2017} may be roughly summarised as follows: given a boundary-defining function $r$, the area of $M\cap \{r>\ve \}$ admits an asymptotic expansion as $\ve\to 0^+$ whose coefficient of $\log(1/\ve)$ is the integral over $\dr_{\infty}M$ of a local conformal invariant. This defines a functional on $d$-dimensional submanifolds of $\dr_\infty X^{d+2}$ that is invariant under conformal transformations. For instance, Graham--Witten \cite{graham1999} computed this functional in dimension $d=2$ and recovered the classical Willmore energy $\int H^2$ when $\dr_\infty X = \R^3$.
	
	From a mathematical viewpoint, one advantage of renormalised volume is that it produces conformally invariant functionals on submanifolds for which smooth minimal submanifolds, and hence their conformal transformations, are critical points, at least when the ambient space is Euclidean. Already in dimension 2, the study of conformal transformations of minimal surfaces is one of the motivations for studying the Willmore energy. In dimension 4 and any codimension, this procedure was applied by Zhang \cite{zhang} and Graham--Reichert \cite{robingraham2020}. Both works yield the following energy when restricted to codimension-one immersions $\Phi \colon \Sigma^4 \to \R^5$:
	\begin{equation*}
		\Egr(\Phi) \coloneqq  \int_\Sigma \left(|\g H|^2_g - H^2 |A|^2_g + 7H^4\right)\, \dd\vol_g,
	\end{equation*}
	where $g\coloneqq \Phi^*\xi$, $\xi$ is the Euclidean metric on $\R^5$, $A$ is the second fundamental form of $\Phi(\Sigma)$, and $H = \frac{1}{4}\tr_g A$ is its mean curvature. However, if the Willmore energy is viewed as a measure of the bending or umbilicity of a given hypersurface, as in the original work of Germain, then the presence of minimal surfaces among its critical points is not essential. In \cite{blitz2023generalized,gover2020}, the authors define a generalised Willmore energy for hypersurfaces $\Sigma^4 \subset \R^5$ as a conformally invariant functional whose Euler--Lagrange equation has a leading term of the form $\lap^2 H$. They use tractor calculus to generate numerous functionals arising in different geometric contexts. 
	
	Another major property of Willmore surfaces in dimension 2 is their correspondence with certain spacelike minimal surfaces in de Sitter space. In 1984, Bryant \cite{bryant1984} established a correspondence between Willmore surfaces in Euclidean space and minimal surfaces in de Sitter space:
	\begin{align*}
		\s^{3,1} \coloneqq \{ x\in \R^5 : |x|^2_\eta = 1 \},\qquad 
		\eta \coloneqq  (\dd x^1)^2 + ... + (\dd x^4)^2 - (\dd x^5)^2 .
	\end{align*}
	He proved that an immersion $\Phi\colon \Sigma^2\to\R^3$ is a Willmore surface (a critical point of the functional that assigns to an immersion the $L^2$ norm of its mean curvature) if and only if its conformal Gauss map $Y\colon \Sigma^2 \to \s^{3,1}$ is a conformal harmonic map, that is, a critical point of the Dirichlet energy that is also a conformal map between Riemannian manifolds, or, in other words, a map whose image is a minimal surface in $\s^{3,1}$ at the points where $Y$ is an immersion. We refer to \cite{marque2021} for a recent introduction. The formula for $Y$ is given in \eqref{eq:def_Y_R5} (see also \eqref{definition_Y_sphere} for an immersion into $\s^5$), and we refer to \cref{sub:intro_CGM} and \cite{dajczer,hertrich-jeromin2003,marque2021} for a precise construction. At a point $x\in\Sigma^2$, the conformal Gauss map $Y(x)\in \s^{3,1}$ represents the sphere in $\R^3$ that is tangent to $\Phi(\Sigma)$ at $\Phi(x)$ and has mean curvature $H(x)$. Since the set of spheres and hyperplanes is invariant under conformal transformations $\Theta\colon \R^3\to \R^3$, one can prove that the conformal Gauss map $Y_{\Theta\circ\Phi}$ of $\Theta\circ\Phi$ is of the form $M_{\Theta}Y_{\Phi}$ for some $M_{\Theta}\in SO(4,1)$, which is an isometry of de Sitter space, see \cref{sub:intro_CGM} for more details. This invariance explains the term \emph{conformal} Gauss map, see Section \ref{sec:Geom-CGM} for a discussion about the geometry of conformal Gauss maps. By comparing the construction of the conformal Gauss map in \cite{dajczer,hertrich-jeromin2003,marque2021} with that of the normal tractor in \cite{curry,vyatkin} when the ambient manifold is conformally flat, one can see that the two constructions are equivalent, see \Cref{sec:tractor}. However, the formalism of the conformal Gauss map is more elementary for the present setting of hypersurfaces in $\R^5$. 
	
	The starting point of Bryant's correspondence is the identity
	\begin{equation*}
		\int_\Sigma |\Arond|^2_{g_\Phi} \dd\vol_{g_\Phi} = \int_\Sigma |\g Y|^2_{g_\Phi} \dd\vol_{g_\Phi} = 2\, \textrm{Area}(Y),
	\end{equation*}	
	where $g_\Phi = \Phi^*\xi$ and $\xi$ is the Euclidean metric on $\R^3$. In this work, we establish an analogue of Bryant's correspondence in dimension $4$. 
	In dimension 2, an important consequence of $Y^*\eta$ being conformal to $\Phi^*\xi$ is that the area of the image of $Y$ is exactly equal to the Dirichlet energy of $Y$. The latter is much more useful in practice because it avoids the degeneracies of $Y^*\eta$ at the umbilic points of $\Phi$. Theorem \ref{thm_duality} shows that, in dimension 4, the area is replaced by the total scalar curvature (one could also call it an Einstein--Hilbert-type functional), and the Dirichlet energy is replaced by a Paneitz-type energy. The Paneitz operator \cite{chang2005} is defined by the following formula: for any $f\in C^\infty(\Sigma;\R)$,
	\begin{equation*}
		P_g f = \lap_g^2 f- \di_g\left[ \left( \frac{2\Scal_g}{3} - 2 \Ric_g \right) \g^g f \right].
	\end{equation*}
	Consequently, we have
	\begin{equation}\label{eq:integral_paneitz_general_formula}
		\int_\Sigma f P_g f\ \dd\vol_g = \int_\Sigma \left( |\lap_g f|^2 +\frac{2\Scal_g}{3} |\g^g f|^2 - 2\Ric_g(\g^g f,\g^g f)\right)\, \dd\vol_g.
	\end{equation}
	A direct computation gives the following relation.
	
	\begin{theoremIntro}\label{thm_duality}
		Let $\Sigma$ be a closed, connected, orientable, smooth four-dimensional manifold. Let $\Phi\colon \Sigma \to \R^5$ be a smooth immersion and $Y \colon \Sigma \to \s^{5,1}$ be its conformal Gauss map. Let $g$ be the first fundamental form of $\Phi$ and define $\pr$ by
		\begin{equation*}
			\pr \coloneqq \frac{1}{4} \int_\Sigma \left( \scal{Y}{P_g Y}_\eta - \frac{4}{3} |\g Y|^4_{g,\eta} - 4\det\Arond\right)\, \dd\vol_g.
		\end{equation*}
		Then we have
		\begin{equation*}
			\Egr = 4\pi^2 \chi(\Sigma) +\pr,
		\end{equation*}
		where $\chi(\Sigma)$ is the Euler characteristic of $\Sigma$ and $\Arond$ is the traceless part of the second fundamental form of $\Phi$. Furthermore, assume that $\det\Arond$ never vanishes. Let $\ve \coloneqq \sign\left[\det_g\Arond\right]$ and $\ggo \coloneqq Y^*\eta$, and define $\Sr$ by
		\begin{equation*}
			\Sr \coloneqq   \frac{\ve}{2} \int_\Sigma \left( 6 - \Scal_\ggo \right) \dd\vol_\ggo,
		\end{equation*}
		where $\Scal_\ggo$ is the scalar curvature of the metric $\ggo$. Then
		\begin{equation*}
			\Egr = 4\pi^2 \chi(\Sigma) + \Sr.
		\end{equation*}
	\end{theoremIntro}
	A variational interpretation is provided in \Cref{subsec:variational_interpretation}. In order to understand the consequences of this result, we make the following remarks.
	\begin{enumerate} 
		\item Let us clarify the sense in which we use the term \emph{duality}. In Bryant's work \cite{bryant1984}, duality refers to the existence of a second Willmore surface which, together with the initial surface, envelopes the same mean-curvature-sphere congruence represented by the conformal Gauss map. We do not construct an analogous dual hypersurface in dimension four here, studying such a dual hypersurface would bring us too far. The analogy established here concerns instead the energy identities underlying Bryant's construction: modulo a topological term, the two-dimensional Willmore energy is both the Dirichlet energy and the area of the conformal Gauss map, whereas Theorem~\ref{thm_duality} expresses the Graham--Reichert energy both as a Paneitz-type energy of $Y$ and as an Einstein--Hilbert-type functional of the induced metric $\ggo=Y^*\eta$. Thus, the term \emph{duality} is used here in this restricted variational sense.
		
		\item  We give a geometric interpretation of the term $\det\Arond$. In dimension 2, the Cayley--Hamilton theorem gives $\det_g\Arond = -\frac{1}{2}|\Arond|^2_g \leq 0$. This is no longer true in dimension $4$, hence the sign plays a role. As shown in Lemma \ref{lm:metric_ggo}, the volume element induced by $Y$ is given by $\dd\vol_\ggo = \big|\det \Arond\big|\, \dd\vol_g$. Thus, the quantity $\int_\Sigma (\det \Arond)\, \dd\vol_g$ can be interpreted as the signed volume of the image of $Y$, counted with multiplicity. The condition $\det\Arond\neq 0$ means that the metric $\ggo$ is a smooth Riemannian metric. 
		
		\item  The metric $\ggo$ may provide a means of constructing further conformal invariants. It would be interesting to characterise particular cases, such as those in which $\ggo$ is Ricci-flat or a critical point of $\Sr$ among the full set of metrics on $\Sigma$, rather than only those arising as first fundamental forms of Gauss maps. 
	\end{enumerate}
	
	The equality between $\Egr$ and $\int_\Sigma \scal{Y}{P_g Y}_\eta$ up to lower-order terms has already been observed in general ambient manifolds in \cite{blitz2023generalized}. However, some computations are computer-assisted. The computations in the present work are all explicit. 
	
	We emphasise a first major difference between dimensions two and four. In dimension 2, the conformal Gauss map $Y\colon (\Sigma^2 ,g_\Phi) \to (\s^{3,1},\eta)$ of an immersion $\Phi\colon \Sigma \to \R^3$ is always conformal because $Y^*\eta = \frac{|\Arond|^2_{g_{\Phi}}}{2}\, g_{\Phi}$. However, the following proposition shows that, in dimension 4, the conformal Gauss map $Y\colon (\Sigma^4 ,\Phi^*\xi) \to (\s^{5,1},\eta)$ is no longer conformal as a map between two Riemannian manifolds in general, which creates additional technical difficulties. Indeed, its induced metric, computed in Lemma \ref{lm:metric_ggo}, is $\ggo \coloneqq Y^*\eta = \Arond^2$.
	
	\begin{propositionIntro}\label{if_Y_conformal}
		Let $\B^4$ be the unit ball of $\R^4$ and $\xi$ be the Euclidean metric on $\R^5$.
		Let $\Phi\colon \B^4 \to \R^5$ be a smooth immersion, and let $Y \colon \B^4 \to \s^{5,1}$ be its conformal Gauss map. Assume that there exists a smooth positive function $\omega\colon \B^4 \to (0,+\infty)$ such that the metrics $g = \Phi^*\xi$ and $\ggo = Y^*\eta$ satisfy $\ggo = \omega^2 g$. 
		
		Then there exist local coordinates $(t^1,\ldots,t^4)$ near any given point of $\B^4$ such that, for any constant $c\in\R^4$ in the range of the coordinates $(t^1,\ldots,t^4)$, the leaves $\Phi(\{t^3=c^3,t^4=c^4\})$ define totally umbilic surfaces in $\R^5$. Similarly, the leaves $\Phi(\{t^1=c^1,t^2=c^2\})$ define totally umbilic surfaces in $\R^5$.
	\end{propositionIntro}
	
	The preceding result can be illustrated using the computations in \cite{robingraham2020}. If $r_1,r_2>0$ satisfy $r_1^2 + r_2^2 = 1$, then the submanifold $\s^2(r_1)\times \s^2(r_2)\subset \s^5$ satisfies
	\begin{equation*}
		\ggo = \Arond^2 = \frac{1}{4}\left( \frac{r_2}{r_1} + \frac{r_1}{r_2} \right)^2 g.
	\end{equation*}
	Hence, the conformal Gauss map of such a hypersurface is indeed conformal. Proposition \ref{if_Y_conformal} shows that the converse imposes geometric restrictions: if $Y$ is conformal as a map between Riemannian manifolds, then the image of $\Phi$ is foliated by two transverse families of totally umbilic surfaces, i.e.\ two-dimensional planes or spheres whose radii need not be constant.

	Graham--Reichert \cite{robingraham2020} investigated the existence of lower or upper bounds for $\Egr$. They showed that $\Egr$ is bounded neither below nor above for immersions of $\s^1\times \s^3$ and $\s^2 \times \s^2$. They also showed in Proposition 1.2 of \cite{robingraham2020} that the round sphere is stable. However, the question of whether $\Egr$ admits a minimum among immersions of $\s^4$, attained by the round sphere, was left open. Here, we give a negative answer.
	\begin{theoremIntro}\label{prop:min_W_spheres}
		Given any closed hypersurface $\Sigma^4 \subset \R^5$, there exists a sequence of smooth immersions $\Phi_k \colon \Sigma \to \R^5$ such that $\Egr(\Phi_k) \xrightarrow[k\to\infty]{}{-\infty}$.
	\end{theoremIntro}
	The construction of such a sequence is based on the observation that the integrand of $\Egr$ is pointwise equal to a negative constant for the isometric immersion of $\R^2 \times \s^2$, see Lemma \ref{lm:II_R2_S2}. It would be interesting to characterise the behaviour of sequences of immersions $(\Phi_k)_{k\in\N}$ such that $\Egr(\Phi_k) \xrightarrow[k\to\infty]{}{-\infty}$. Since $\R^2 \times\s^2$ has constant mean curvature, other noncompact constant-mean-curvature hypersurfaces may provide further examples. More generally, one may ask which hypersurfaces can be used to construct such sequences. We emphasise that, in dimension 2, there is no conformally invariant functional that is the integral of a polynomial in the second fundamental form and is unbounded both below and above. 
	
	However, using the explicit formula for $P_{g_\Phi} Y$, where $Y$ is the conformal Gauss map of an immersion $\Phi \colon \Sigma\to \R^5$, we obtain the following lower bound.
	\begin{propositionIntro}\label{prop:min_P}
		Let $\Sigma^4$ be a smooth, closed, connected, orientable manifold. Let $\Phi \colon \Sigma \to \R^5$ be a smooth immersion, and set $g \coloneqq \Phi^*\xi$, where $\xi$ is the Euclidean metric on $\R^5$. Let $Y$ be the conformal Gauss map of $\Phi$. Then
		\begin{align*}
			\Ep(\Phi) \coloneqq  \int_{\Sigma} \scal{Y}{P_g Y}_\eta \dd\vol_g =\int_\Sigma \left( 4|\g^g H|^2 + \frac{1}{3}|\Arond|^4_g + 2H^2 |\Arond|^2_g -4H\tr_g(\Arond^3)  + 2\tr_g(\Arond^4)\right)\, \dd\vol_g.
		\end{align*}
		Furthermore, we have $\Ep(\Phi) \geq 0$ with equality if and only if $\Phi$ is a round sphere.
	\end{propositionIntro}
	
	The construction of $\Ep$ is close to that of the two-dimensional Willmore energy in the sense that the Paneitz operator and the Laplacian are both GJMS operators considered in the critical dimension. From an analytic viewpoint, the energy $\Ep$ is simpler than $\Egr$ for two reasons. First, $\Ep$ is quadratic in $Y$, whereas $\Egr$ is not. Second, $\Ep$ is positive, which suggests minimising it among immersions of a fixed closed four-dimensional manifold into $\R^5$. A minimiser, if it exists, would provide a notion of distinguished immersion of that manifold into $\R^5$, up to conformal transformations. However, minimal surfaces do not appear to be critical points of $\Ep$.
	
	In view of Proposition \ref{prop:min_P}, we can ask questions similar to those arising from the Willmore conjecture in dimension $2$. Can we minimise $\Ep$ among immersions of a given closed hypersurface $\Sigma$ in $\R^5$? Can we compute the infimum of $\Ep$ among immersions of $\Sigma$ into $\R^5$?

	\textbf{Structure of the paper:} In \Cref{se:Preliminaries}, we introduce the hypersurface notation and construct the conformal Gauss map. In \Cref{se:ggo}, we compute the induced metric, the Levi--Civita connection, the normal frame, and the second fundamental form of the conformal Gauss map. In \Cref{se:proof_duality}, we prove the scalar-curvature and Paneitz identities in Theorem \ref{thm_duality} and discuss their variational interpretation. In \Cref{se:no_conformal}, we prove Proposition \ref{if_Y_conformal}. In \Cref{se:energy_bounds}, we contrast the unboundedness of $\Egr$ with the positivity of $\Ep$ by proving Theorem \ref{prop:min_W_spheres} and Proposition \ref{prop:min_P}. In \Cref{sec:cayley_ham}, we give the algebraic identities derived from the Cayley--Hamilton theorem, and in \Cref{sec:curvature_Phi}, we record the required curvature identities for Euclidean hypersurfaces.
	
	\textbf{Acknowledgements:} I would like to thank Paul Laurain for his constant support and advice. I am grateful to Yann Bernard, Rod Gover, and Frédéric Hélein for many useful discussions. I thank the referees for their important comments.
	
	\textbf{Remark:} After the first publication of this article, Graham--Reichert informed me that they had also obtained an unpublished proof of Theorem \ref{prop:min_W_spheres}.
	
	\section{Immersions in $\R^5$}\label{se:Preliminaries}
	
	\subsection{Fundamental forms and notations}\label{sub:Willmore}
	
	The following notions of the first and second fundamental forms can be found in \cite{dC1992}.
	
	Let $\Sigma$ be a closed, orientable, four-dimensional manifold, and let $\Phi \colon \Sigma\to \R^5$ be a smooth immersion. We denote by $\xi$ the Euclidean metric on $\R^5$. Let $g \coloneqq  \Phi^*\xi$. Let $n$ be the Gauss map of $\Phi$, $A$ its second fundamental form, $H$ its mean curvature, and $\Arond$ the traceless part of the second fundamental form. They are given by the formulae
	\begin{align*}
		A_{ij} &\coloneqq  -\scal{\dr_i \Phi}{\dr_j n}_\xi, \\
		H &\coloneqq  \frac{1}{4} \tr_g (A) = \frac{1 }{4} \scal{\lap_g\Phi}{n}_\xi, \\
		\Arond &= A - Hg,\\
		\g n &= -H\g \Phi - \Arond\g \Phi.
	\end{align*}
	We consider the energy
	\begin{equation*}
		\Egr(\Phi) \coloneqq \int_\Sigma \left( |\g H|_g^2 - H^2 |\Arond|^2_g + 3H^4 \right)\, \dd\vol_{g}.
	\end{equation*}
	To simplify the computations, we will often use normal coordinates for the first fundamental form $g$ at a fixed point $x_0\in\Sigma$. Since we compute geometric quantities, the results do not depend on the choice of these coordinates; the coordinates merely shorten the intermediate computations.
	
	The Christoffel symbols of the first fundamental form $g$ are given by $(\Gamma^g)_{ij}^k = \scal{\dr_{ij}^2\Phi}{\dr^k \Phi}_\xi$, where the index $k$ is raised using the metric $g$. Therefore, in normal coordinates around $x_0\in \Sigma$ for the metric $g$, we have $(\Gamma^g)_{ij}^k(0)=0$ for all $i,j,k$, and thus $\dr^2_{ij}\Phi(x_0)\perp\dr_k\Phi(x_0)$. This implies that all the second derivatives $\dr^2_{ij} \Phi(x_0)$ are normal to $\dd\Phi_{x_0}(T_{x_0}\Sigma)$, which is a four-dimensional linear subspace of $\R^5$. Thus, all the second derivatives $\dr^2_{ij} \Phi(x_0)$ are proportional to $n(x_0)$ when the chosen coordinates at $x_0$ satisfy $g_{ij}(0)=\delta_{ij}$ and $\Gamma_{ij}^k(0)=0$, with $0$ denoting the coordinates of $x_0$.
	
	We will frequently use one of the following assumptions concerning the degeneracy of $\Arond$:
	\begin{align}\label{hyp:no_umbilic}
		\det_g \Arond \neq 0\qquad  \text{on }\Sigma,
	\end{align}
	or
	\begin{align}\label{hyp:small_umbilic}
		\{\det_g \Arond \neq 0\} \text{ is dense in }\Sigma.
	\end{align}
	
	\subsection{The conformal Gauss map}\label{sub:intro_CGM}
	
	An introduction to conformal Gauss maps can be found in \cite[Sections 1.6 and 3.4]{hertrich-jeromin2003}. A geometric picture for surfaces can be found in \cite[Sections 2.2 and 2.3]{marque2021}. Here we sketch the construction for immersions in $\s^5$ in \Cref{sec:CGM-S}, in $\R^5$ in \Cref{sec:CGM-R}. We discuss its geometry in \Cref{sec:Geom-CGM} and its relation to tractor calculus in \Cref{sec:tractor}. 
	
	\subsubsection{Construction of the conformal Gauss map for an immersion into $\s^5$}\label{sec:CGM-S}
	Let $\Sigma^4$ be a smooth manifold, and let $\Psi \colon \Sigma^4\to \s^5$ be a smooth immersion with mean curvature $H$ and normal vector $N$. Let $x\in\Sigma$ and consider the 4-sphere $S_x\subset \s^5$ that is tangent to $\Psi(\Sigma)$ at $\Psi(x)$ with mean curvature $H(x)$. Let $\hat{S}_x\subset \R^6$ be the 5-sphere that intersects $\s^5$ orthogonally along $S_x$. If $H(x)=0$, then $\hat{S}_x$ is a flat 5-plane. If $H(x)\neq 0$, the centre of $\hat{S}_x$ is $c_x \coloneqq  \Psi(x) + \frac{N(x)}{H(x)}$, and its radius is $\frac{1}{|H(x)|}$. To obtain a description that remains valid when the mean curvature vanishes, we first consider the homogeneous vector
	\begin{equation}\label{homogeneous_Y}
		\left[ \begin{pmatrix}
			\Psi(x) + \frac{N(x)}{H(x)}\\ 1
		\end{pmatrix} \right] \in \RP^6 .
	\end{equation}
	If $\eta = (\dd x^1)^2 + \cdots + (\dd x^6)^2 - (\dd x^7)^2$ is the Minkowski metric on $\R^7$ and $|\cdot|_\eta$ is the associated norm, we have
	\begin{equation*}
		\left| \begin{pmatrix}
			\Psi(x) + \frac{N(x)}{H(x)}\\ 1
		\end{pmatrix} \right|^2_\eta = \frac{1}{H(x)^2} >0.
	\end{equation*}
	We normalise the homogeneous vector in \eqref{homogeneous_Y} so that it has unit norm and obtain the following representative of $S_x$:
	\begin{equation}\label{definition_Y_sphere}
		Y(x) \coloneqq  \begin{pmatrix}
			H(x)\Psi(x) + N(x)\\ H(x)
		\end{pmatrix} = H(x) \begin{pmatrix} \Psi(x)\\ 1 \end{pmatrix} + \begin{pmatrix}
			N(x)\\ 0
		\end{pmatrix}.
	\end{equation}
	This definition remains valid when $H$ vanishes. By construction, $Y(x)$ belongs to the de Sitter space
	\begin{equation*}
		\s^{5,1} \coloneqq  \{ p\in \R^7 : |p|^2_\eta = 1 \}.
	\end{equation*}
	We denote the space $(\R^7,\eta)$ by $\R^{6,1}$. We call $Y$ the \emph{conformal Gauss map} because it is a conformally invariant version of the Gauss map, in the sense that, for any $\Xi \in \Conf(\s^5)$, there exists a matrix $M_{\Xi}\in SO(6,1)$ such that the conformal Gauss map $Y_{\Xi\circ\Psi}$ of $\Xi\circ\Psi$ satisfies $Y_{\Xi\circ \Psi} = M_{\Xi} Y$.
	
	\medskip 
	\subsubsection{Construction of the conformal Gauss map for an immersion into $\R^5$}\label{sec:CGM-R}
	Consider the stereographic projection from the point $Q= (0,\ldots,0,1)$
	\begin{equation*}
		\pi \colon \left\{ \begin{array}{r c l}
			\s^5\setminus \{ Q \} & \to & \R^5 \\
			\begin{pmatrix}
				x_1\\
				\vdots\\
				x_6
			\end{pmatrix} & \mapsto & \frac{1}{1-x_6} \begin{pmatrix}
				x_1 \\ \vdots \\ x_5
			\end{pmatrix}
		\end{array} 
		\right. ,\qquad \qquad \pi^{-1} \colon \left\{ \begin{array}{r c l}
			\R^5 & \to & \s^5\setminus \{Q\} \\
			\begin{pmatrix}
				x_1\\
				\vdots\\
				x_5
			\end{pmatrix} & \mapsto & \frac{1}{1+x_1^2 +\cdots + x_5^2} \begin{pmatrix}
				2x_1\\
				\vdots \\ 
				2x_5 \\
				x_1^2 + \cdots +x_5^2 -1 
			\end{pmatrix}
		\end{array} .
		\right.
	\end{equation*}
	The differential of $\pi^{-1}$ is given by
	\begin{equation*}
		\dd\pi^{-1}(z) = \frac{2}{1+|z|^2} \begin{pmatrix}
			\dd z\\ 0
		\end{pmatrix} - \frac{4\scal{z}{\dd z}}{(1+|z|^2)^2} \begin{pmatrix} z\\ -1 \end{pmatrix}.
	\end{equation*}
	Hence, its norm is given by
	\begin{equation*}
		|\dd\pi^{-1}(z)|^2 = \frac{4}{(1+|z|^2)^2}|\dd z|^2.
	\end{equation*}
	Let $\Psi \colon \Sigma^4 \to \s^5\setminus \{Q\}$ be an immersion, set $\Phi \coloneqq \pi \circ\Psi$, and let $n$ be the normal vector of $\Phi$. Then
	\begin{equation}\label{eq:Psi-Phi}
		\Psi = \frac{1}{1+|\Phi|^2_\xi} \begin{pmatrix}
			2\Phi\\ |\Phi|^2_\xi - 1
		\end{pmatrix}.
	\end{equation}
	We compute the mean curvature $H_\Psi$ and the normal vector $N$ of $\Psi$ in terms of the mean curvature $H_\Phi$ and the normal vector $n$ of $\Phi$. The vector $N$ is given by
	\begin{equation*}
		N = \frac{\dd\pi^{-1}(n)}{|\dd\pi^{-1}(n)|} = \begin{pmatrix}
			n \\ 0
		\end{pmatrix} - \frac{2\scal{n}{\Phi}_\xi }{1+|\Phi|^2} \begin{pmatrix}\Phi\\ -1\end{pmatrix}.
	\end{equation*}
	By \cite[Proposition 1.2.1]{kuwert2012}, we have
	\begin{equation*}
		H_\Psi = \frac{1+|\Phi|^2}{2}\left( H_\Phi +2\frac{\scal{\Phi}{n}_\xi}{1+|\Phi|^2} \right).
	\end{equation*}
	Therefore, the conformal Gauss map of $\Phi$ is given by
	\begin{equation}\label{eq:def_Y_R5}
		Y_\Phi \coloneqq H_\Phi \begin{pmatrix}
			\Phi\\ \frac{|\Phi|^2-1}{2} \\ \frac{|\Phi|^2+1}{2}
		\end{pmatrix}+ \begin{pmatrix}
			n\\ \scal{n}{\Phi}_\xi \\ \scal{n}{\Phi}_\xi
		\end{pmatrix}.
	\end{equation}
	\medskip
	\subsubsection{Derivatives of the conformal Gauss map}
	We define the vector
	\begin{equation}\label{eq:def_nu}
		\nu \coloneqq \begin{pmatrix}
			\vspace{0.2em}\Phi\\ \vspace{0.2em}\frac{|\Phi|^2-1}{2} \\ \vspace{0.2em}\frac{|\Phi|^2+1}{2}
		\end{pmatrix}.
	\end{equation}
	Since $|\nu|^2_\eta = 0$, we have, in particular, $\scal{\nu}{\g \nu}_\eta = 0$. The derivatives of $Y$ are given by
	\begin{equation*}
		\g Y = (\g H)\begin{pmatrix}
			\vspace{0.2em} \Phi\\ \vspace{0.2em}\frac{|\Phi|^2-1}{2} \\ \vspace{0.2em}\frac{|\Phi|^2+1}{2}
		\end{pmatrix} + H\begin{pmatrix}
			\g \Phi\\ \scal{\Phi}{\g \Phi}_\xi \\ \scal{\Phi}{\g \Phi}_\xi
		\end{pmatrix} + \begin{pmatrix}
			\g n\\ \scal{\g n}{\Phi}_\xi \\ \scal{\g n}{\Phi}_\xi
		\end{pmatrix}.
	\end{equation*}
	This simplifies to a formula that we will often use:
	\begin{equation}\label{eq:derivativeY}
		\g Y = (\g H)\begin{pmatrix}
			\vspace{0.2em}\Phi\\ \vspace{0.2em}\frac{|\Phi|^2-1}{2} \\ \vspace{0.2em}\frac{|\Phi|^2+1}{2}
		\end{pmatrix} -\Arond \begin{pmatrix}
			\g \Phi\\ \scal{\Phi}{\g \Phi}_\xi \\ \scal{\Phi}{\g \Phi}_\xi
		\end{pmatrix} = (\g H)\nu - \Arond \g\nu.
	\end{equation}
	We have the orthogonality relations 
	\begin{equation*}
		\scal{Y}{\nu}_\eta = 0, \qquad \text{ and }\qquad  \scal{\g Y}{\nu}_\eta = 0.
	\end{equation*}
	
	\medskip
	\subsubsection{Geometry of the conformal Gauss map}\label{sec:Geom-CGM}
	
	In this paragraph, we explain the connections between the geometry of an immersion $\Phi\colon \Sigma^4\to \R^5$ and its conformal Gauss map $Y\colon \Sigma^4\to \s^{5,1}$. This is mainly a reformulation of the discussion in \cite{marque2021}.
	
	By stereographic projection (see \eqref{eq:Psi-Phi}), working with $\Phi\colon \Sigma\to \R^5$ is equivalent to working with $\Psi\colon \Sigma\to \s^5$. The sphere $\s^5$ can be viewed as the following slice of the light cone:
	\begin{equation*}
		S \coloneqq \left\{ p\in\R^{6,1} : |p|^2_{\eta}=0,\ p^7=1 \right\}.
	\end{equation*}
	Under this identification, the immersion $\Psi$ is described by the immersion $\begin{pmatrix}
		\Psi\\ 1
	\end{pmatrix}\in S$, already used in \eqref{definition_Y_sphere}. This vector is lightlike and provides an identification between immersed submanifolds in $\s^5$ and immersed submanifolds in $S$. As briefly mentioned above, there is an identification between $\Conf(\s^5)$ and $SO(6,1)$, see for instance \cite[Section 1.3]{hertrich-jeromin2003}. Under this identification, the action of a conformal transformation $\Xi\in \Conf(\s^5)$ can be described in terms of a matrix $M_{\Xi}\in SO(6,1)$ as follows. If we denote by $c^7$ the seventh and last entry of the vector $M_{\Xi}\begin{pmatrix}
		\Psi\\ 1
	\end{pmatrix}$, then we have
	\begin{equation*}
		\begin{pmatrix}
			\Xi\circ\Psi\\ 1
		\end{pmatrix} = \frac{1}{c^7}\, M_{\Xi}\begin{pmatrix}
			\Psi\\ 1
		\end{pmatrix}.
	\end{equation*}
	On the other hand, the set of hyperspheres and hyperplanes in $\R^5$ and $\s^5$ can be parametrised by $\s^{5,1}$ \cite[Section 1.1]{hertrich-jeromin2003}. Considering the congruence of generalised spheres that are tangent to $\Psi$ and have the same mean curvature yields the notion of the conformal Gauss map in \eqref{definition_Y_sphere}, see \cite[Lemma and Definition 3.4.1]{hertrich-jeromin2003}. We explained the term \textit{conformal} Gauss map above. The expression "Gauss map" refers to the fact that $Y$ is a canonical choice of vector field lying in the normal bundle of the immersion $\begin{pmatrix}
		\Psi\\ 1
	\end{pmatrix}$ from $\Sigma$ into the light cone of $\R^{6,1}$. Returning from $\Psi$ to $\Phi$ yields the formula \eqref{eq:def_Y_R5}. If we regard $Y$ as an immersion (restricted to the domain where it is an immersion) into $\s^{5,1}$, then we obtain a spacelike immersion of codimension 2. Hence, its normal bundle can be described using two lightlike normal vector fields. One is given by $\begin{pmatrix}
		\Psi\\ 1
	\end{pmatrix}$, or equivalently by $\nu$ defined in \eqref{eq:def_nu} when working with immersions in $\R^5$. The other will be computed in \Cref{se:ggo} and denoted by $\nu^*$ for immersions in $\R^5$. It can also be chosen as a map of the form $\begin{pmatrix}
		\Psi^*\\ 1
	\end{pmatrix}$, when it is well-defined. This is what Bryant called the dual Willmore surface in \cite{bryant1984}. 
	
	Understanding the geometry of $Y$ in this higher-dimensional context is one of the main motivations for the present work. Once we have explicit formulae for both vectors $\nu$ and $\nu^*$, we can compute the second fundamental form of $Y$ and hence the intrinsic curvature of its induced metric using the Gauss--Codazzi formulae. In order to recover the functional $\Egr$, one needs to find quantities of the same order as the derivatives of the second fundamental form of $\Phi$. The formula \eqref{eq:derivativeY} shows that the derivatives of $H$ in the computations of the derivatives of $Y$ occur only along $\nu$, which is lightlike. Therefore, they will not contribute to the computation of geometric quantities induced by $Y$. Hence, we need to differentiate once more and thus obtain the second fundamental form, which also determines the curvature of the first fundamental form of $Y$ through the Gauss--Codazzi equations. One of the simplest scalar quantities is the scalar curvature. This is the main source of the computations presented in this manuscript.
	
	\medskip
	\subsubsection{Relation with tractor calculus}\label{sec:tractor}
	Let us explain more precisely the relation between the conformal Gauss map and the normal tractor. We use the notation of \cite{curry}. In the Euclidean metric $\xi$, the standard tractor bundle of $(\R^5,[\xi])$ splits as $\mathcal T\simeq \R\oplus T^*\R^5\oplus\R$, and a tractor is represented by a triple $\begin{pmatrix}\sigma\\ \mu\\ \rho\end{pmatrix}$. In this splitting, the tractor metric is given by the following formula, see \cite[Sections~3.4--3.6]{curry}:
	\[
	h\left( \begin{pmatrix}\sigma\\ \mu\\ \rho\end{pmatrix}, \begin{pmatrix}\sigma\\ \mu\\ \rho\end{pmatrix} \right) =|\mu|_\xi^2+2\sigma\rho.
	\]
	The tractor connection is given by the following formula in local coordinates:
	\[
	\nabla_a^{\mathcal T} \begin{pmatrix}\sigma\\ \mu_b \\ \rho\end{pmatrix} 
	=
	\begin{pmatrix} 
		\dr_a\sigma-\mu_a\\[2mm]
		\dr_a \mu_b+\rho\, \delta_{ab} \\[2mm]
		\dr_a \rho 
	\end{pmatrix} .
	\]
	We now describe explicitly the resulting parallel trivialisation. For $z\in\R^5$, set
	\[
	\nu_0(z)\coloneqq
	\begin{pmatrix}
		z\\[1mm]
		\frac{|z|^2-1}{2}\\[2mm]
		\frac{|z|^2+1}{2}
	\end{pmatrix},
	\qquad
	e_-\coloneqq
	\begin{pmatrix}
		0_{\R^5}\\ -1\\ -1
	\end{pmatrix}.
	\]
	The following map is an isometry from the tractor fibre at $z$ to $\R^{6,1}$:
	\[
	\mathcal I_z \begin{pmatrix} \sigma\\ \mu \\ \rho \end{pmatrix} 
	\coloneqq
	\sigma e_-+\dd\nu_0(z)\cdot \mu^\sharp+\rho\, \nu_0(z)
	\in\R^{6,1}
	\]
	Indeed, we have 
	\[
	\scal{\nu_0(z)}{e_-}_\eta=1,
	\qquad \text{ and }
	|\nu_0(z)|_\eta^2=|e_-|_\eta^2=0.
	\]
	Moreover, for any vectors $v,w\in \R^5$, we have
	\[
	\scal{\dd\nu_0(z)\cdot v}{\dd\nu_0(z)\cdot w }_\eta =\scal{v}{w}_\xi.
	\]
	Hence, a direct computation shows that $\eta\left( \mathcal{I}_z(\cdot),\mathcal{I}_z(\cdot)\right) = h_z$ for each $z\in\R^5$. That is, the map $\mathcal{I}_z\colon \R^{6,1}\to \mathcal{T}_z$ is an isometry on each fibre.
	Moreover, direct differentiation gives
	\[
	\dr_a\left(\mathcal I\begin{pmatrix} \sigma\\ \mu \\ \rho \end{pmatrix}  \right)
	=
	\mathcal I\left(
	\nabla_a^{\mathcal T}\begin{pmatrix} \sigma\\ \mu \\ \rho \end{pmatrix}
	\right).
	\]
	Hence, $\mathcal I$ is a parallel trivialisation of $\mathcal T$.
	
	Along the immersion $\Phi$, the normal tractor is represented in the Euclidean splitting by $N_\Phi=\begin{pmatrix} 0 \\ n \\ H_\Phi \end{pmatrix}$, see \cite[Section 6.1]{curry}, with the sign convention that the second fundamental form in \cite{curry} is $-A$ in our convention. Applying the parallel trivialisation to the normal tractor gives
	\begin{equation*}
		\mathcal I_{\Phi}(N_\Phi)
		=\dd\nu_0(\Phi)(n)+H_\Phi\nu_0(\Phi) =
		\begin{pmatrix}
			n\\
			\scal{n}{\Phi}_\xi\\
			\scal{n}{\Phi}_\xi
		\end{pmatrix}
		+
		H_\Phi
		\begin{pmatrix}
			\Phi\\
			\frac{|\Phi|^2-1}{2}\\[1mm]
			\frac{|\Phi|^2+1}{2}
		\end{pmatrix} =Y_\Phi.
	\end{equation*}
	Thus, under the parallel trivialisation of the standard tractor bundle of the conformally flat ambient space, the normal tractor is exactly the conformal Gauss map defined in \eqref{eq:def_Y_R5}.

	\section{Geometry of the conformal Gauss map}\label{se:ggo}
	
	In this section, we compute the first and second fundamental forms of the conformal Gauss map of an immersion into $\R^5$, together with the geometric quantities required in the proof of Theorem \ref{thm_duality}. We fix a smooth immersion $\Phi \colon \Sigma^4 \to \R^5$. Let $g \coloneqq \Phi^*\xi$ be the first fundamental form of $\Phi$, and let $n$, $A$, and $H$ denote its unit normal vector, second fundamental form, and mean curvature, respectively. Let $\eta$ be the Lorentz metric on $\R^{6,1}$. Let $Y\colon \Sigma\to \s^{5,1}$ be the conformal Gauss map of $\Phi$, as defined in \eqref{eq:def_Y_R5}, and let $\ggo \coloneqq Y^*\eta$.
	
	\subsection{The induced metric and volume form}\label{subsec:induced_metric}
	
	In the following lemma, we compute the metric $\ggo$. It also shows that, under the assumption \eqref{hyp:no_umbilic}, the map $Y$ is a smooth immersion.
	\begin{lemma}\label{lm:metric_ggo}
		The metric $\ggo$ is given by $\ggo \coloneqq Y^*\eta = \Arond^2$.
	\end{lemma}
	\begin{proof}
		By \eqref{eq:derivativeY}, we have
		\begin{align*}
			\scal{\dr_i Y}{\dr_j Y}_\eta &= \scal{\Arond_i^{\ k}\, \dr_k \Phi }{ \Arond_j^{\ l}\, \dr_l \Phi}_\xi = \Arond_i^{\ k}\Arond_{jk}.
		\end{align*}
		Here the indices are raised with respect to $g$.
	\end{proof}
	Throughout the remainder of this section, we work under the assumption \eqref{hyp:no_umbilic}.
	
	In particular, the volume form of $\ggo$ is given by $\dd\vol_\ggo = |\det\Arond|\, \dd\vol_g$. Hence, the 4-form $(\det \Arond) \dd\vol_g$ that appears in Theorem \ref{thm_duality} can be regarded as a signed volume form of the manifold $Y(\Sigma)$. Furthermore, we note that $\sign[\det\Arond]$ corresponds to the orientation of the basis $\left( Y,\dr_1 Y,...,\dr_4 Y, \nu , \begin{pmatrix}
		\vspace{0.2em} 0_{\R^5} \\ -1 \\ -1
	\end{pmatrix} \right)$, as shown by the following computation.
	
	\begin{align*}
		\det \left[ Y,\dr_1 Y,...,\dr_4 Y, \nu , \begin{pmatrix}
			\vspace{0.2em}	0_{\R^5} \\ -1 \\ -1
		\end{pmatrix} \right] &= \det \left[ \begin{pmatrix}
			n\\ \scal{n}{\Phi} \\ \scal{n}{\Phi}
		\end{pmatrix} ,\begin{pmatrix}
			\Arond\g \Phi \\ \scal{\Phi}{\Arond \g \Phi} \\ \scal{\Phi}{\Arond\g \Phi}
		\end{pmatrix} , \begin{pmatrix}
			\vspace{0.2em}\Phi \\ \vspace{0.2em}\frac{|\Phi|^2-1}{2} \\ \vspace{0.2em}\frac{|\Phi|^2+1}{2}
		\end{pmatrix}  ,  \begin{pmatrix}
			0_{\R^5} \\ -1\\ -1
		\end{pmatrix} \right] \\[2mm]
		&= \det \left[ \begin{pmatrix}
			n\\ 0 \\ 0
		\end{pmatrix} ,\begin{pmatrix}
			\Arond\g \Phi \\ 0 \\ 0
		\end{pmatrix} , \begin{pmatrix}
			\Phi \\ -\frac{1}{2} \\ \frac{1}{2}
		\end{pmatrix}  ,  \begin{pmatrix}
			0_{\R^5} \\ -1\\ -1
		\end{pmatrix} \right] \\[2mm]
		&= \det \left[ \begin{pmatrix}
			n\\ 0 \\ 0
		\end{pmatrix} ,\begin{pmatrix}
			\Arond\g \Phi \\ 0 \\ 0
		\end{pmatrix} , \begin{pmatrix}
			\Phi \\ -1 \\ 0
		\end{pmatrix}  ,  \begin{pmatrix}
			0_{\R^5} \\ -1\\ -1
		\end{pmatrix} \right] \\[2mm]
		&= \det \left[ \begin{pmatrix}
			n\\ 0 
		\end{pmatrix} ,\begin{pmatrix}
			\Arond\g \Phi \\ 0 
		\end{pmatrix} , \begin{pmatrix}
			\Phi \\ -1
		\end{pmatrix}   \right] \\[2mm]
		&= \det[\Arond\g \Phi,n] = \det\Arond.
	\end{align*}
	
	\subsection{The Levi--Civita connection of $\ggo$}\label{subsec:connection_ggo}
	
	To compute the curvature of $Y(\Sigma)$, we first compute its Christoffel symbols.
	\begin{lemma}
		The Christoffel symbols of $\ggo$ are given by 
		\begin{align*}
			(\Gamma^\ggo)_{ij}^k \coloneqq \scal{\dr^2_{ij} Y}{\dr_l Y}_\eta \ggo^{kl} &= \left( \Arond^{-1} \right)^{kb}\left( \g^g_i \Arond_{jb} \right) + (\g^\ggo H)^k \Arond_{ij} - \left( \Arond^{-1} \right)^k_{\ i} (\dr_j H) + (\Gamma^g)^k_{ij},
		\end{align*}
		where $\left( \Arond^{-1} \right)^k_{\ i} \coloneqq \left( \Arond^{-1} \right)^{km} g_{mi}$.
	\end{lemma}
	
	\begin{remark}
		This expression is symmetric in $i$ and $j$. By the Gauss--Codazzi identity, we have
		\begin{align*}
			\g^g_i \Arond_{jb} &= \g^g_i\left( A_{jb} - H g_{jb} \right)\\[2mm]
			&= \g^g_j A_{ib} - (\dr_i H) g_{jb} \\[2mm]
			&= \g^g_j \Arond_{ib} + (\dr_j H)g_{ib} - (\dr_i H)g_{jb}.
		\end{align*}
		Hence, we have
		\begin{equation*}
			\g^g_i \Arond_{jb} -(\dr_j H)g_{ib} = \g^g_j \Arond_{ib}- (\dr_i H)g_{jb}.
		\end{equation*}
	\end{remark}

	\begin{proof}
		Let $\nu \coloneqq \begin{pmatrix} \vspace{0.2em}\Phi\\ \vspace{0.2em}\frac{|\Phi|^2-1}{2} \\ \vspace{0.2em}\frac{|\Phi|^2+1}{2} \end{pmatrix}$. Using \eqref{eq:derivativeY}, we compute
		\begin{align*}
			\scal{\dr^2_{ij} Y}{\dr_k Y}_\eta =& \scal{ \dr_i\left[ (\dr_j H) \nu -\Arond_j^{\ l} \dr_l\nu \right]  }{ (\dr_k H) \nu -\Arond_k ^{\ b} \dr_b \nu }_\eta \\[2mm]
			=& \scal{ (\dr^2_{ij} H) \nu
				+(\dr_j H)\dr_i\nu }{   (\dr_k H) \nu -\Arond_k ^{\ b} \dr_b \nu   }_\eta  - \scal{  \left( \dr_i \Arond_j^{\ l} \right)  \dr_l \nu
				+\Arond_j^{\ l} \dr^2_{il}\nu }{  (\dr_k H) \nu -\Arond_k ^{\ b} \dr_b \nu   }_\eta.
		\end{align*}
		Using the relations $\nu\perp \nu,\g\nu$, $\scal{\nu}{\dr^2_{ab}\nu}_\eta= -\scal{\dr_a\nu}{\dr_b\nu}_\eta$, and $\scal{\dr_a\nu}{\dr_b\nu} = g_{ab}$, we obtain
		\begin{align*}
			\scal{\dr^2_{ij} Y}{\dr_k Y}_\eta =& -(\dr_j H)\Arond_k ^{\ b} g_{ib} + \left( \dr_i \Arond_j^{\ l} \right) \Arond_k^{\ b} g_{lb} + (\dr_k H) \Arond_j^{\ l} g_{il} + \Arond_j^{\ l} \Arond_k^{\ b} \scal{\dr^2_{il} \Phi}{\dr_b \Phi}_\xi  \\[2mm]
			=& (\dr_k H) \Arond_{ij} - (\dr_j H)\Arond_{ki} + \Arond_k^{\ b} \left[ \left(\dr_i \Arond_j^{\ l} \right) g_{lb} + \Arond_j^{\ l}\scal{\dr^2_{il} \Phi}{\dr_b \Phi}_\xi  \right].
		\end{align*}
		We now consider the last term:
		\begin{align*}
			\left(\dr_i \Arond_j^{\ l} \right) g_{lb} + \Arond_j^{\ l}\scal{\dr^2_{il} \Phi}{\dr_b \Phi}_\xi &= \dr_i \left( \Arond_{jb} \right) - \Arond_j^{\ l} (\dr_i g_{lb})+ \Arond_j^{\ l}\scal{\dr^2_{il} \Phi}{\dr_b \Phi}_\xi \\[2mm]
			&= \dr_i \left( \Arond_{jb} \right) - \Arond_j ^{\ l} \scal{\dr_l \Phi}{\dr^2_{ib} \Phi}_\xi \\[2mm]
			&=  \dr_i \left( \Arond_{jb} \right) - \Arond_{jl} (\Gamma^g)_{ib}^l.
		\end{align*}
		If $\g^g$ denotes the Levi-Civita connection of $(\Sigma,g)$ and $(\Gamma^g)_{ij}^k$ are the Christoffel symbols of $g$, then
		\begin{align*}
			\g^g_i \Arond_{jb} &= \dr_i \left( \Arond_{jb} \right) - \Arond_{jl} (\Gamma^g)_{ib}^l - \Arond_{bl}(\Gamma^g)_{ij}^l \\[2mm]
			& = \left( \dr_i \Arond_j^{\ l} \right)g_{lb} + \Arond_j^{\ l}\scal{\dr^2_{il} \Phi}{\dr_b \Phi}_\xi- \Arond_{bl}(\Gamma^g)_{ij}^l.
		\end{align*}
		Let $(\Gamma^\ggo)_{ijk} := \scal{\dr_{ij}Y}{\dr_k Y}_\eta$ denote the Christoffel symbols of the first kind of $\ggo$. Then
		\begin{equation*}
			\Gamma^\ggo_{ijk} = \left( \g^g_i \Arond_{jb} \right)\Arond_k^{\ b} + (\dr_k H)\Arond_{ij} - (\dr_j H)\Arond_{ki} + \left( \Arond^2\right)_{kl} (\Gamma^g)_{ij}^l .
		\end{equation*}
		We obtain
		\begin{align*}
			(\Gamma^\ggo)_{ij}^k &= \ggo^{kl}\, \Gamma^\ggo_{ijl} \\[2mm]
			&= \left( \Arond^{-1} \right)^{km} g_{mn} \left( \Arond^{-1} \right)^{nl} \left[ \left( \g^g_i \Arond_{jb} \right)\Arond_{lc}g^{bc} + (\dr_l H)\Arond_{ij} - (\dr_j H)\Arond_{li} + \left( \Arond^2\right)_{l\lambda} (\Gamma^g)_{ij}^\lambda \right] \\[2mm]
			&= \left( \Arond^{-1} \right)^{kb}\left( \g^g_i \Arond_{jb} \right) + (\g^\ggo H)^k \Arond_{ij} - \left( \Arond^{-1} \right)^k_{\ i} (\dr_j H) + (\Gamma^g)^k_{ij}.
		\end{align*}
	\end{proof}
	
	\subsection{A null normal frame and the second fundamental form}\label{subsec:second_ff_Y}
	
	We now compute the second fundamental form $B$ of $Y(\Sigma)\subset \s^{5,1}$. Consider the basis $(\dr_1 Y,...,\dr_4 Y)$ of $TY(\Sigma)$, which we extend to a basis $(\dr_1 Y,...,\dr_4 Y, \nu,\nu^*)$ of $T\s^{5,1} = T Y(\Sigma)\oplus TY(\Sigma)^\perp$, with $\nu = \begin{pmatrix} \vspace{0.2em}\Phi\\  \vspace{0.2em}\frac{|\Phi|^2 - 1}{2} \\ \vspace{0.2em}\frac{|\Phi|^2 + 1}{2} \end{pmatrix}$ and $|\nu|^2_\eta = |\nu^*|^2_\eta = 0$. The second fundamental form $B$ is then given by
	\begin{equation*}
		B = B^\nu \nu + B^* \nu^* = \text{proj}_{TY(\Sigma)^\perp}(\g^2 Y) = \frac{\scal{\g^2 Y}{\nu}_\eta }{ \scal{\nu}{\nu^*}_\eta} \nu^* + \frac{\scal{\g^2 Y}{\nu^*}_\eta }{ \scal{\nu}{\nu^*}_\eta} \nu,
	\end{equation*}
	where, for $i,j\in\inter{1}{4}$, we write
	\begin{align*}
		B^\nu_{ij} \coloneqq \frac{\scal{\dr^2_{ij} Y}{\nu^*}_\eta }{ \scal{\nu}{\nu^*}_\eta},\qquad \text{ and }\qquad B^*_{ij} \coloneqq \frac{\scal{\dr^2_{ij} Y}{\nu}_\eta }{ \scal{\nu}{\nu^*}_\eta}.
	\end{align*}
	In particular, its squared norm is given by
	\begin{align}
		|B|^2_\eta &= 2\scal{\nu}{\nu^*}_\eta \left( \ggo^{-1} \right)^{ij} \left( \ggo^{-1} \right)^{kl} B^\nu_{ik} B^*_{jl}  \nonumber \\[2mm]
		&= \frac{2}{\scal{\nu}{\nu^*}_\eta} \left( \ggo^{-1} \right)^{ij} \left( \ggo^{-1} \right)^{kl} \scal{\dr^2_{ik} Y}{\nu^*}_\eta \scal{\dr^2_{jl} Y}{\nu}_\eta. \label{eq:Bsquare}
	\end{align}
	To compute $B$, we first derive a formula for $\nu^*$. In the intermediate computations, we assume that $H^2 + |\g^g H|^2_g \neq 0$. Since the computation is local, if $H^2+ |\g^g H|^2_g = 0$ in a neighbourhood of some point $x_0\in\Sigma$, we can smoothly approximate $\Phi$ near $x_0$ by immersions satisfying $H^2 + |\g^g H|^2_g \neq 0$ in a neighbourhood of $x_0$. The final result (see Corollary \ref{second_ff_Y}) is independent of whether $H^2 + |\g^g H|^2_g$ vanishes, hence the formula for the approximating immersions passes to the limit. We now seek an expression for $\nu^*$ of the form $ \begin{pmatrix} \vspace{0.2em}\Phi^*\\  \vspace{0.2em}\frac{|\Phi^*|^2 - 1}{2} \\ \vspace{0.2em}\frac{|\Phi^*|^2 + 1}{2} \end{pmatrix}$. 
	
	\begin{lemma}\label{lm:computation_Phistar}
		$\Phi^*$ is given by 
		\begin{equation*}
			\Phi^* = \Phi + \frac{2H}{H^2 + |\g^\ggo H|^2_\ggo} n - \frac{2 \left(\Arond^{-1} \right)^{ki} \dr_i H}{H^2 + |\g^\ggo H|^2_\ggo}  \dr_k \Phi .
		\end{equation*}
	\end{lemma}

	\begin{proof}
		We seek $\Phi^*$ in the form
		\begin{equation*}
			\Phi^* = \Phi + \beta^k \dr_k \Phi + \gamma n.
		\end{equation*}
		Let $\nu^* \coloneqq \begin{pmatrix}  \vspace{0.2em}\Phi^*\\ \vspace{0.2em} \frac{|\Phi^*|^2 - 1}{2} \\ \vspace{0.2em}\frac{|\Phi^*|^2 + 1}{2} \end{pmatrix}$. The following equations follow from the fact that $\nu^*$ belongs to the normal space of $Y$:
		\begin{equation}\label{relation_Phistar}
			\scal{Y}{\nu^*}_\eta =0,\qquad \text{ and }\qquad \scal{\g Y}{\nu^*}_\eta = 0.
		\end{equation}
		The first equation of \eqref{relation_Phistar} can be written as
		\begin{align*}
			0 &=\scal{ \begin{pmatrix}
					\vspace{0.2em}H\Phi + n\\
					\vspace{0.2em}H \frac{|\Phi|^2 - 1}{2} + \scal{n}{\Phi} \\
					\vspace{0.2em}H \frac{|\Phi|^2 + 1}{2} + \scal{n}{\Phi} 
			\end{pmatrix} }{ \begin{pmatrix}
					\vspace{0.2em}\Phi^*\\
					\vspace{0.2em}\frac{|\Phi^*|^2-1}{2} \\
					\vspace{0.2em}\frac{|\Phi^*|^2+1}{2} 
			\end{pmatrix} }_\eta \\
			&= \scal{ H\Phi + n}{\Phi^*}_\xi + \left( H \frac{|\Phi|^2 - 1}{2} + \scal{n}{\Phi} \right)\frac{|\Phi^*|^2-1}{2} - \left( H \frac{|\Phi|^2 + 1}{2} + \scal{n}{\Phi}  \right)\frac{|\Phi^*|^2+1}{2}   \\
			&= \scal{ H\Phi + n}{\Phi^*}_\xi - H\frac{|\Phi^*|^2}{2} -\left( H \frac{|\Phi|^2}{2} +\scal{n}{\Phi} \right) .
		\end{align*}
		We now substitute $\Phi^*= \Phi + \beta^k \dr_k \Phi + \gamma n$:
		\begin{align}
			0 &= \scal{ H\Phi + n}{\Phi + \beta^k \dr_k \Phi + \gamma n}_\xi - H\frac{|\Phi + \beta^k \dr_k \Phi + \gamma n|^2}{2} -\left( H \frac{|\Phi|^2}{2} +\scal{n}{\Phi} \right)  \nonumber\\
			&=  H|\Phi|^2 + H\beta^k \scal{\Phi}{\dr_k \Phi} + ( \gamma H+1) \scal{\Phi}{n} + \gamma \nonumber\\
			&\quad - \frac{H}{2}\left( |\Phi|^2 + |\beta|^2_g + \gamma^2 + 2\beta^k \scal{\Phi}{\dr_k \Phi} + 2\gamma \scal{\Phi}{n} \right) \nonumber\\
			&\quad -\left( H \frac{|\Phi|^2}{2} +\scal{n}{\Phi} \right)  \nonumber\\
			&=  \gamma - \frac{ H}{2}(\gamma^2 + |\beta|^2_g) . \label{relation1}
		\end{align}
		We now focus on the second equation in \eqref{relation_Phistar}. For each $i\in\inter{1}{4}$, we have
		\begin{align*}
			0 =& \scal{ (\dr_i H) \begin{pmatrix} \vspace{0.2em}\Phi\\ \vspace{0.2em}\frac{|\Phi|^2-1}{2} \\ \vspace{0.2em}\frac{|\Phi|^2+1}{2} \end{pmatrix} -\Arond_i^{\ l} \begin{pmatrix}
					\dr_l \Phi\\ \scal{\dr_l \Phi}{\Phi} \\ \scal{\dr_l \Phi}{\Phi}
			\end{pmatrix} }{ \begin{pmatrix}
					\vspace{0.2em}	\Phi^*\\
					\vspace{0.2em}	\frac{|\Phi^*|^2-1}{2} \\
					\vspace{0.2em}	\frac{|\Phi^*|^2+1}{2} 
			\end{pmatrix} }_\eta \\
			=& \scal{ (\dr_i H) \Phi - \Arond_i^{\ l}\dr_l \Phi}{\Phi^*} \\
			&+ \left[ (\dr_i H) \frac{|\Phi|^2-1}{2} - \Arond_i^{\ l} \scal{\dr_l \Phi}{\Phi} \right] \frac{|\Phi^*|^2-1}{2} \\
			&- \left[ (\dr_i H) \frac{|\Phi|^2+1}{2} - \Arond_i^{\ l} \scal{\dr_l \Phi}{\Phi} \right] \frac{|\Phi^*|^2+1}{2} \\
			=&  \scal{ (\dr_i H) \Phi - \Arond_i^{\ l}\dr_l \Phi}{\Phi^*} -  (\dr_i H) \frac{|\Phi^*|^2}{2} -\left[ (\dr_i H) \frac{|\Phi|^2}{2} - \Arond_i^{\ l} \scal{\dr_l \Phi}{\Phi} \right].
		\end{align*}
		We substitute $\Phi^*= \Phi + \beta^k \dr_k \Phi + \gamma n$:
		\begin{equation*}
			0 = \scal{ (\dr_i H) \Phi - \Arond_i^{\ l}\dr_l \Phi}{\Phi + \beta^k \dr_k \Phi + \gamma n} - (\dr_i H) \frac{|\Phi + \beta^k \dr_k \Phi + \gamma n|^2}{2} -\left[ (\dr_i H) \frac{|\Phi|^2}{2} - \Arond_i^{\ l} \scal{\dr_l \Phi}{\Phi} \right].
		\end{equation*}
		We obtain
		\begin{align}
			0 =& (\dr_i H) \left( |\Phi|^2 + \beta^k \scal{\Phi}{\dr_k \Phi} + \gamma \scal{\Phi}{n} \right) - \lambda \Arond_i^{\ l} \left( \scal{\dr_l \Phi}{\Phi} + \beta^k g_{kl} \right)v \nonumber\\
			&- \frac{\dr_i H}{2} \left( |\Phi|^2 + |\beta|^2_g + \gamma^2 + 2\beta^k \scal{\Phi}{\dr_k \Phi} + 2\gamma \scal{\Phi}{n} \right) \nonumber\\
			&-\left[ (\dr_i H) \frac{|\Phi|^2}{2} - \Arond_i^{\ l} \scal{\dr_l \Phi}{\Phi} \right] \nonumber \\
			=& -\frac{\dr_i H}{2} (|\beta|^2_g + \gamma^2 ) -  \Arond_{ik} \beta^k . \label{relation2}
		\end{align}
		We deduce from \eqref{relation2} that
		\begin{equation*}
			0 = \lambda \left[ \gamma\frac{\dr_i H}{H} + \Arond_{ik} \beta^k \right].
		\end{equation*}
		Under the assumption \eqref{hyp:no_umbilic}, we obtain an expression for $\beta^k$ in terms of $\gamma$:
		\begin{equation*}
			\beta^k = -\gamma \left(\Arond^{-1} \right)^{ki}\frac{\dr_i H}{H}.
		\end{equation*}
		Thus, the squared norm $|\beta|^2_g$ is given by
		\begin{equation*}
			|\beta|^2_g = g_{kl} \beta^k \beta^l = g_{kl} \left(\Arond^{-1} \right)^{ki} \left(\Arond^{-1} \right)^{lj} \frac{(\dr_i H)(\dr_j H)}{H^2}\gamma^2.
		\end{equation*}
		Substituting this into \eqref{relation1}, we obtain an equation involving only $\gamma$:
		\begin{equation*}
			0 = \gamma -\frac{1}{2}\left( H + g_{kl} \left(\Arond^{-1} \right)^{ki} \left(\Arond^{-1} \right)^{lj} \frac{(\dr_i H)(\dr_j H)}{H} \right) \gamma^2.
		\end{equation*}
		The solution $\gamma=0$ yields $\Phi^* = \Phi$. The solution $\gamma\neq 0$ is given by
		\begin{equation*}
			\gamma = \frac{2H}{H^2 + g_{kl} \left(\Arond^{-1} \right)^{ki} \left(\Arond^{-1} \right)^{lj}(\dr_i H)(\dr_j H)} = \frac{2H}{H^2 + |\g^\ggo H|^2_\ggo}.\qedhere
		\end{equation*}
	\end{proof}
	
	The final ingredient needed to compute $B$ is the following lemma.
	
	\begin{lemma}
		We have
		\begin{align*}
			\scal{\nu}{\nu^*}_\eta &= \frac{-2}{H^2 + |\g^\ggo H|^2_\ggo }.
		\end{align*}
	\end{lemma}
	
	\begin{proof}
		We compute
		\begin{align*}
			\scal{\nu}{\nu^*}_\eta =& \scal{ \begin{pmatrix}
					\vspace{0.2em}\Phi\\
					\vspace{0.2em}\frac{|\Phi|^2-1}{2} \\
					\vspace{0.2em}\frac{|\Phi|^2+1}{2}
			\end{pmatrix} }{ \begin{pmatrix}
					\vspace{0.2em}\Phi^*\\
					\vspace{0.2em}\frac{|\Phi^*|^2-1}{2}\\
					\vspace{0.2em}\frac{|\Phi^*|^2+1}{2}
			\end{pmatrix} }_\eta \\
			=& \scal{\Phi}{\Phi^*}_\xi + \frac{|\Phi|^2-1}{2}\frac{|\Phi^*|^2-1}{2} - \frac{|\Phi|^2+1}{2}\frac{|\Phi^*|^2+1}{2} \\
			=& \scal{\Phi}{\Phi^*}_\xi - \frac{|\Phi|^2}{2} - \frac{|\Phi^*|^2}{2}.
		\end{align*}
		Using Lemma \ref{lm:computation_Phistar}, we obtain
		\begin{align*}
			\scal{\nu}{\nu^*}_\eta =& \scal{ \Phi}{ \Phi + \frac{2H}{H^2 + |\g^\ggo H|^2_\ggo} n - \frac{2 \left(\Arond^{-1} \right)^{ki} \dr_i H}{H^2 + |\g^\ggo H|^2_\ggo}  \dr_k \Phi } - \frac{|\Phi|^2}{2} - \frac{1}{2} \left| \Phi + \frac{2H}{H^2 + |\g^\ggo H|^2_\ggo} n - \frac{2 \left(\Arond^{-1} \right)^{ki} \dr_i H}{H^2 + |\g^\ggo H|^2_\ggo}  \dr_k \Phi \right|^2.
		\end{align*}
		Expanding, we obtain
		\begin{align*}
			\scal{\nu}{\nu^*}_\eta	=&  |\Phi|^2 +  \frac{2H \scal{\Phi}{n} }{H^2 + |\g^\ggo H|^2_\ggo} - \frac{2 \scal{\Phi}{\dr_k \Phi} \left(\Arond^{-1} \right)^{ki} \dr_i H}{H^2 + |\g^\ggo H|^2_\ggo} -\frac{|\Phi|^2}{2} \\
			& - \frac{1}{2}\Bigg( |\Phi|^2 + \left( \frac{2H}{H^2 + |\g^\ggo H|^2_\ggo} \right)^2 + \left(\frac{2}{H^2 + |\g^\ggo H|^2_\ggo} \right)^2 |\Arond^{-1} \g H|^2_g  + \frac{4H \scal{\Phi}{n} }{H^2 + |\g^\ggo H|^2_\ggo} - \frac{4 \scal{\Phi}{\dr_k \Phi} \left(\Arond^{-1} \right)^{ki} \dr_i H}{H^2 + |\g^\ggo H|^2_\ggo} \Bigg). 
		\end{align*}
		Therefore,
		\begin{equation*}
			\scal{\nu}{\nu^*}_\eta= -2\frac{H^2 + |\Arond^{-1} \g^g H|^2_g }{\left( H^2 + |\g^\ggo H|^2_\ggo \right)^2 } = \frac{-2}{H^2 + |\g^\ggo H|^2_\ggo}.\qedhere
		\end{equation*}
		
	\end{proof}
	
	\paragraph{Notation.} Set $f \coloneqq  H^2 + |\g^\ggo H|^2_\ggo$.
	
	By \eqref{eq:Bsquare}, we have
	\begin{align*}
		|B|^2_\eta &= -f\left( \ggo^{-1} \right)^{ij} \left( \ggo^{-1} \right)^{kl} \scal{\dr^2_{ik} Y}{\nu^*}_\eta \scal{\dr^2_{jl} Y}{\nu}_\eta.
	\end{align*}
	We first compute $B^*$.
	\begin{lemma}
		The second fundamental form in the direction $\nu^*$ is given by
		\begin{align}\label{hessianY_direction_nu}
			\scal{\dr^2_{ij} Y}{\nu}_\eta &= \Arond_{ij}.
		\end{align}
	\end{lemma}
	
	\begin{proof}
		Using \eqref{eq:derivativeY} and the orthogonality relations $\nu\perp \nu,\g \nu$, we compute
		\begin{equation*}
			\scal{\dr^2_{ij} Y}{\nu}_\eta = \scal{ (\dr^2_{ij} H) \nu	+(\dr_j H)\dr_i\nu - \left( \dr_i \Arond_j^{\ l} \right)  \dr_l \nu
				-\Arond_j^{\ l} \dr^2_{il}\nu }{ \nu }_\eta = \Arond_j^{\ l}\scal{\dr_i\nu}{\dr_l\nu}_\eta = \Arond_j^{\ l} g_{il}.\qedhere
		\end{equation*} 
	\end{proof}
	
	We next compute $B^\nu$.
	\begin{lemma}
		The second fundamental form in the direction $\nu$ is given by
		\begin{align*}
			\scal{\dr^2_{ij} Y}{\nu^*}_\eta &= \frac{-2}{f} \Big[ (\g^\ggo)^2_{ij} H + H\ggo_{ij} \Big] - \Arond_{ij}.
		\end{align*}
	\end{lemma}
	
	\begin{proof}
		A direct computation gives
		\begin{align*}
			& \scal{\dr^2_{ij} Y}{\nu^*}_\eta \\[2mm]
			&= \scal{ (\dr^2_{ij} H) \begin{pmatrix} \vspace{0.2em}\Phi\\
					\vspace{0.2em}\frac{|\Phi|^2-1}{2} \\
					\vspace{0.2em}\frac{|\Phi|^2+1}{2} 
				\end{pmatrix}
				+(\dr_j H)\begin{pmatrix} \dr_i \Phi\\ \scal{\dr_i \Phi}{\Phi} \\ \scal{\dr_i \Phi}{\Phi} \end{pmatrix} - \left( \dr_i \Arond_j^{\ l} \right)  \begin{pmatrix}
					\dr_l \Phi\\ \scal{\dr_l \Phi}{\Phi} \\ \scal{\dr_l \Phi}{\Phi}
				\end{pmatrix} 
				-\Arond_j^{\ l} \begin{pmatrix}
					\dr^2_{il}\Phi \\ 
					\scal{\dr^2_{il}\Phi}{\Phi} + \scal{\dr_l\Phi}{\dr_i \Phi}\\
					\scal{\dr^2_{il}\Phi}{\Phi} + \scal{\dr_l\Phi}{\dr_i \Phi}
			\end{pmatrix} }{ \begin{pmatrix}
					\vspace{0.2em}\Phi^*\\
					\vspace{0.2em}\frac{|\Phi^*|^2-1}{2} \\
					\vspace{0.2em}\frac{|\Phi^*|^2+1}{2} 
			\end{pmatrix} }_\eta \\[2mm]
			&= (\dr^2_{ij} H) \scal{\nu}{\nu^*} + (\dr_j H) \left[ \scal{\dr_i \Phi}{\Phi^*} - \scal{\dr_i\Phi}{\Phi} \right] \\[2mm]
			&\quad - \left( \dr_i \Arond_j^{\ l} \right) \left[ \scal{\dr_l\Phi}{\Phi^*} - \scal{\dr_l \Phi}{\Phi} \right] - \Arond_j^{\ l} \left[ \scal{\dr^2_{il}\Phi}{\Phi^*} - \scal{\dr^2_{il} \Phi}{\Phi} - g_{il} \right].
		\end{align*}
		Write the formula in Lemma \ref{lm:computation_Phistar} as $\Phi^* = \Phi + \beta^k \dr_k \Phi + \gamma n$. We obtain
		\begin{align}
			\scal{\dr^2_{ij} Y}{\nu^*}_\eta		&= (\dr^2_{ij} H) \frac{-2}{f} + (\dr_j H) \beta^k g_{ik}  - \left(\dr_i \Arond_j ^{\ l} \right) \beta^k g_{kl} - \Arond_j^{\ l} ( \beta^k\scal{\dr_{il}^2 \Phi}{\dr_k \Phi} + \gamma \scal{\dr^2_{il} \Phi}{n} - g_{il} ) \nonumber\\
			&= (\dr^2_{ij} H) \frac{-2}{f} + (\dr_j H) \beta^k g_{ik}  - \left(\dr_i \Arond_j ^{\ l} \right) \beta^k g_{kl} - \Arond_j^{\ l} ( \beta^k \scal{\dr^2_{li} \Phi}{\dr_k \Phi}_\xi + \gamma A_{il} - g_{il} ) . \label{eq:Bnu_1}
		\end{align}
		Note that $\Arond_j^{\ l} A_{il} = \left(\Arond^2 \right)_{ij} + H \Arond_{ij} = \ggo_{ij} + H \Arond_{ij}$. Moreover,
		\begin{align*}
			\left( \dr_i \Arond_j^{\ l} \right)  g_{lk} &= \dr_i \left( \Arond_j^{\ l}g_{lk} \right) - \Arond_j^{\ l}(\dr_i g_{lk} ).
		\end{align*}
		We rewrite the terms involving derivatives of $\Arond$ in \eqref{eq:Bnu_1} as follows:
		\begin{align*}
			\left( \dr_i \Arond_j^{\ l} \right) \beta^k g_{lk}+ \beta^k \Arond_j^{\ l} \scal{\dr^2_{li} \Phi}{\dr_k \Phi}_\xi &= \dr_i\left( \Arond_{jk} \right) \beta^k - \beta^k \Arond_j^{\ l} \scal{\dr^2_{ik} \Phi}{\dr_l \Phi}_\xi \\
			&=  \dr_i\left( \Arond_{jk} \right) \beta^k - \beta^k \Arond_{jl} (\Gamma^g)^l_{ik} \\
			&= \beta^k\left[ \g^g_i \Arond_{jk} + \Arond_{k\lambda} (\Gamma^g)^\lambda_{ij} \right].
		\end{align*}
		Therefore, \eqref{eq:Bnu_1} becomes
		\begin{align}
			\scal{\dr^2_{ij} Y}{\nu^*}_\eta &= -\frac{2}{f} (\dr^2_{ij} H) + (\dr_j H) \beta^k g_{ik} -\beta^k \left[ \g^g_i \Arond_{jk} + \Arond_{k\lambda} (\Gamma^g)^\lambda_{ij} \right] -\gamma \left( \ggo_{ij} + H\Arond_{ij} \right) + \Arond_{ij} \nonumber \\
			&= -\frac{2}{f} (\dr^2_{ij} H) + (\dr_j H) \beta^k g_{ik} -\beta^k \left[ \g^g_i \Arond_{jk} + \Arond_{k\lambda} (\Gamma^g)^\lambda_{ij} \right] -\gamma \ggo_{ij} +(1- \gamma H) \Arond_{ij}.  \label{eq:Bnu_2}
		\end{align}
		By Lemma \ref{lm:computation_Phistar}, we have
		\begin{align*}
			\beta^k = \frac{-2(\dr_l H) \left( \Arond^{-1} \right)^{lk} }{f}, \qquad\text{ and }\qquad \gamma = \frac{2H}{f}.
		\end{align*}
		Substituting these expressions into \eqref{eq:Bnu_2}, we obtain
		\begin{align}
			& \scal{\dr^2_{ij} Y}{\nu^*}_\eta \nonumber \\[2mm]
			=& -\frac{2}{f} (\dr^2_{ij} H) + (\dr_j H)  \frac{-2(\dr_l H) \left( \Arond^{-1} \right)^{lk} }{f}  g_{ik} - \frac{-2(\dr_l H) \left( \Arond^{-1} \right)^{lk} }{f}  \left[ \g^g_i \Arond_{jk} + \Arond_{k\lambda} (\Gamma^g)^\lambda_{ij} \right]  -\frac{2H}{f} \ggo_{ij} +(1- \frac{2H^2}{f}) \Arond_{ij} \nonumber \\[2mm]
			=& \frac{-2}{f} \left[ (\dr^2_{ij} H)+ (\dr_j H)(\dr_l H) \left(\Arond^{-1} \right)^{lk} g_{ik} - (\dr_l H) \left( \Arond^{-1} \right)^{lk} \g^g_i \Arond_{jk} - (\dr_l H) (\Gamma^g)_{ij}^l + H \ggo_{ij} \right] + \frac{f-2H^2}{f} \Arond_{ij}. \label{eq:Bnu_32}
		\end{align}
		We rewrite the term $(\dr_l H)\left( \Arond^{-1} \right)^{lk}$ as follows:
		\begin{align*}
			(\dr_l H)\left( \Arond^{-1} \right)^{lk}g_{jk} &= (\dr_l H)\left( \Arond^{-1} \right)^{lk}g_{qk} \left( \Arond^{-1} \right)^{qp} \Arond_{pj} = (\g^\ggo H)^p \Arond_{pj}.
		\end{align*}
		Substituting this identity into \eqref{eq:Bnu_32}, we obtain
		\begin{align}
			\scal{\dr^2_{ij} Y}{\nu^*}_\eta &= \frac{-2}{f} \left[ (\dr^2_{ij} H) + (\dr_j H) (\g^\ggo H)^p \Arond_{pi} - (\dr_l H) \left( \Arond^{-1} \right)^{lk} \g^g_i \Arond_{jk} - (\dr_l H) (\Gamma^g)_{ij}^l + H \ggo_{ij} \right]  + \frac{f-2H^2}{f} \Arond_{ij} . \label{eq:Bnu_4}
		\end{align}
		The Hessian of $H$ with respect to the metric $\ggo$ is given by
		\begin{align*}
			(\g^\ggo)^2_{ij} H &= \dr^2_{ij} H - (\Gamma^\ggo)^k_{ij} (\dr_k H) \\
			&= \dr^2_{ij} H - \left[ (\Gamma^g)_{ij}^k + \left( \Arond^{-1} \right)^{kb}\left( \g^g_i \Arond_{jb} \right) + (\g^\ggo H)^k \Arond_{ij} - \left( \Arond^{-1} \right)^{km}g_{mi} (\dr_j H)\right] (\dr_k H).
		\end{align*}
		Hence, we can write \eqref{eq:Bnu_4} as
		\begin{align*}
			\scal{\dr^2_{ij} Y}{\nu^*}_\eta	=& \frac{-2}{f} \Bigg( (\g^\ggo)^2_{ij} H + \left[ (\Gamma^g)_{ij}^k + \left( \Arond^{-1} \right)^{kb}\left( \g^g_i \Arond_{jb} \right) + (\g^\ggo H)^k \Arond_{ij} - \left( \Arond^{-1} \right)^{km}g_{mi} (\dr_j H)\right] (\dr_k H)\\
			&+ (\dr_j H) (\g^\ggo H)^p \Arond_{pi} - (\dr_l H) \left( \Arond^{-1} \right)^{lk} \g^g_i \Arond_{jk} - (\dr_l H) (\Gamma^g)_{ij}^l + H \ggo_{ij} \Bigg] + \frac{f-2H^2}{f} \Arond_{ij}.
		\end{align*}
		Since $(\g^\ggo H)^k \Arond_{ij}(\dr_k H) = |\g^\ggo H|^2_\ggo \Arond_{ij}$, we obtain
		\begin{equation*}
			\scal{\dr^2_{ij} Y}{\nu^*}_\eta = \frac{-2}{f} \left[ (\g^\ggo)^2_{ij} H + H\ggo_{ij}  \right] +\frac{f- 2H^2 - 2|\g^\ggo H|^2_\ggo}{f} \Arond_{ij} 
			= \frac{-2}{f} \left[ (\g^\ggo)^2_{ij} H + H\ggo_{ij} \right] - \Arond_{ij}. \qedhere
		\end{equation*}
	\end{proof}
	
	We now complete the description of $B$.	
	\begin{corollary}\label{second_ff_Y}
		We have
		\begin{align*}
			B_{ij} &= \frac{-f}{2}\Arond_{ij} \nu^* + \Bigg( (\g^\ggo)^2_{ij} H + H\ggo_{ij} +\frac{f}{2} \Arond_{ij} \Bigg) \nu,  \\
			B^*_{ij} &=\frac{-f}{2}\Arond_{ij}, \\
			B^\nu_{ij} &= (\g^\ggo)^2_{ij} H + H\ggo_{ij} +\frac{f}{2} \Arond_{ij}.
		\end{align*} 
	\end{corollary}
	
	\begin{proof}
		The preceding computations give
		\begin{equation*}
			B_{ij} = \frac{-f}{2} \left[ \Arond_{ij} \nu^* + \Bigg( \frac{-2}{f} \left[ (\g^\ggo)^2_{ij} H + H\ggo_{ij} \right] - \Arond_{ij} \Bigg) \nu \right] .\qedhere 
		\end{equation*} 
	\end{proof}
	\subsection{Mean curvature and quadratic contractions}\label{subsec:mean_curvature_Y}
	
	We now compute the mean curvature vector $\vec{b}$ of $Y(\Sigma)\subset \s^{5,1}$.
	\begin{lemma}\label{lm:mean_curv_vectorY}
		The mean curvature vector $\vec{b}$ of $Y$ is given by
		\begin{equation*}
			4\vec{b} = \frac{-f}{2}\tr_\ggo\left(\Arond \right) \nu^* +\left[  \lap_\ggo H + 4H + \frac{f}{2} \tr_\ggo\left( \Arond \right) \right] \nu.
		\end{equation*}
	\end{lemma}
	
	\begin{proof}
		The mean curvature vector is given by
		\begin{equation*}
			4\vec{b} = \tr_\ggo B = \left(\ggo^{-1} \right)^{ij} B_{ij}.\qedhere
		\end{equation*}
	\end{proof}
	
	We can now compute the squared norms.
	
	\begin{lemma}\label{lm:normsB}
		The squared norms of $B$ and $\vec{b}$ are given by
		\begin{align*}
			|B|^2_{\ggo,\eta} &= 2\left( \scal{\Arond}{(\g^\ggo)^2 H}_\ggo + H\tr_\ggo\left( \Arond \right) \right) + f\left|\Arond \right|^2_\ggo, \\
			16\,|\vec{b}|^2_\eta &=2\left( \tr_\ggo\Arond \right) \left( \lap_\ggo H + 4H \right) + f \left( \tr_\ggo\Arond \right)^2.
		\end{align*}
	\end{lemma}
	
	\begin{proof}
		We compute the squared norm of $B$:
		\begin{align*}
			|B|^2_{\ggo,\eta} &= 2\scal{\nu}{\nu^*}_\eta\left( \ggo^{-1} \right)^{ij} \left( \ggo^{-1} \right)^{kl} B^\nu_{ik} B^*_{jl} \\[2mm]
			&= 2\left( \ggo^{-1} \right)^{ij} \left( \ggo^{-1} \right)^{kl} \Arond_{ik} \left( (\g^\ggo)^2_{jl} H + H\ggo_{jl} \right) + f \left( \ggo^{-1} \right)^{ij} \left( \ggo^{-1} \right)^{kl}\Arond_{ik} \Arond_{jl} \\[2mm]
			&= 2\left( \scal{\Arond}{(\g^\ggo)^2 H}_\ggo + H\tr_\ggo\left( \Arond \right) \right) + f\left|\Arond \right|^2_\ggo .
		\end{align*}
		We compute the squared norm of $4\vec{b}$:
		\begin{align*}
			16\, |\vec{b}|^2_\eta &= -f\scal{\nu}{\nu^*}_\eta \tr_\ggo\left(\Arond\right) \left[ \lap_\ggo H + 4H + \frac{f}{2} \tr_\ggo\left( \Arond \right) \right] \\[2mm]
			&= 2\left( \tr_\ggo\Arond \right) \left( \lap_\ggo H+ 4H \right) + f \left( \tr_\ggo\Arond \right)^2.
		\end{align*}
	\end{proof}
	
	\section{Energy identities for the conformal Gauss map}\label{se:proof_duality}
	
	In this section, let $\Sigma^4$ be a closed, orientable $4$-manifold, let $\Phi \colon \Sigma^4 \to \R^5$ be a smooth immersion, and let $Y\colon\Sigma\to \s^{5,1}$ be its conformal Gauss map. We define $g\coloneqq \Phi^*\xi$, where $\xi$ is the flat metric on $\R^5$.
	
	\subsection{Scalar curvature}
	
	In this subsection, we work only under the assumption \eqref{hyp:no_umbilic}. We prove the equality $\Egr = 4\pi^2 \chi(\Sigma) + \Sr$ stated in Theorem \ref{thm_duality}. We define $\ggo \coloneqq Y^*\eta$, where $\eta$ is the Minkowski metric on $\s^{5,1}$.
	
	The Gauss equation can be written as follows, see for instance \cite[Theorem II.2.1]{chavel2006}. Given any vector fields $x,y,z,t$ tangent to $Y(\Sigma)$, we have
	\begin{equation*}
		\Riem^\ggo(x,y,z,t) = \Riem^\eta(x,y,z,t) + \scal{B(x,z)}{B(y,t)}_\eta - \scal{B(x,t)}{ B(y,z)}_\eta.
	\end{equation*}
	The Riemann tensor of $\s^{5,1}$ is given by
	\begin{equation*}
		\Riem^\eta(x,y,z,t) = \scal{x}{z}_\eta \scal{y}{t}_\eta - \scal{x}{t}_\eta \scal{y}{z}_\eta.
	\end{equation*}
	From this, we compute the scalar curvature of $Y(\Sigma)$.
	\begin{lemma}
		The scalar curvature $\Scal_\ggo$ of $Y(\Sigma)\subset (\s^{5,1},\eta)$ is given by
		\begin{equation*}
			\Scal_\ggo = 12 -  |B|^2_\eta + 16 \big|\vec{b}\big|^2_\eta .
		\end{equation*}
	\end{lemma}
	
	\begin{remark}
		The term $12$ arises because we assume that $\ggo$ is a nondegenerate Riemannian metric, as stipulated in \eqref{hyp:no_umbilic}, and because we work in dimension $\dim \Sigma=4$. 
	\end{remark}
	
	\begin{proof}
		Here, we raise and lower indices with respect to the metric $\ggo$. We compute
		\begin{equation*}
			\begin{aligned} 
				\Ric^\ggo_{ik} &= {\Riem^\ggo}_{ijk}^{\ \ \ \ j} \\[2mm]
				&= {\Riem^\eta}_{ijk}^{\ \ \ \ j} + \scal{B_{ik} }{B_j^{\ j} }_\eta - \scal{B_i^{\ j} }{ B_{jk} }_\eta \\[2mm]
				&=  \ggo_{ik}\ggo_j^{\ j} - \ggo_i^{\ j}\ggo_{jk} + \scal{B_{ik} }{4 \vec{b} }_\eta - \scal{B_i^{\ j} }{ B_{jk} }_\eta  \\[2mm]
				&= 3\ggo_{ik} + 4\scal{B_{ik} }{\vec{b}}_\eta - \scal{B_i^{\ j} }{ B_{jk} }_\eta  .
			\end{aligned} 
		\end{equation*}
		Contracting once more, we obtain
		\begin{equation*}
			\Scal_\ggo = {\Ric^\ggo }_i^{\ i} = 12 + 16 |\vec{b}|^2_\eta - |B|^2_{\ggo,\eta}.
		\end{equation*}
	\end{proof}
	
	Integrating the scalar curvature, we obtain the following.
	\begin{proposition}\label{lm:integral_ScalY}
		Let $\ve\coloneqq \sign\left[\det_g\Arond\right]$, which is constant by \eqref{hyp:no_umbilic} and the connectedness of $\Sigma$. Then
		\begin{equation*}
			\int_\Sigma \Scal_\ggo\, \dd\vol_\ggo = 12\,\vol_\ggo(\Sigma) - \ve\int_\Sigma \left( 2|\g^g H|^2_g + H^2|\Arond|^2_g - 2H\left(\tr_g\Arond^3 \right)\right)  \dd\vol_g.
		\end{equation*}
	\end{proposition}

	\begin{proof}
		By Lemma \ref{lm:normsB}, we directly obtain
		\begin{equation*}
			|B|^2_{\ggo,\eta} - 16|\vec{b}|^2_\eta = 2\left[ \scal{\Arond}{(\g^\ggo)^2 H}_\ggo  - \left(\tr_\ggo \Arond \right) \left( \lap_\ggo H +3H \right) \right] + f \left( \left|\Arond\right|^2_\ggo - \left( \tr_\ggo \Arond \right)^2 \right).
		\end{equation*}
		We integrate with respect to $d\vol_\ggo$:
		\begin{align}
			\int_\Sigma \left( |B|^2_{\ggo,\eta} - 16|\vec{b}|^2_\eta\right) \dd\vol_\ggo &= 2\int_\Sigma \left(\scal{\Arond}{ (\g^\ggo)^2 H}_\ggo - \left(\tr_\ggo \Arond \right) \lap_\ggo H \right) \dd\vol_\ggo \label{eq:integral_scal1}\\[2mm]
			&\quad - 6\int_\Sigma (\tr_\ggo \Arond) H\ \dd\vol_\ggo + \int_\Sigma f\left( |\Arond|^2_\ggo - (\tr_\ggo \Arond)^2  \right) \dd\vol_\ggo. \label{eq:integral_scal2}
		\end{align}
		We integrate the first integral by parts:
		\begin{align}
			\int_\Sigma \left( \scal{\Arond}{ (\g^\ggo)^2 H}_\ggo - \left(\tr_\ggo \Arond \right) \lap_\ggo H \right) \dd\vol_\ggo &= \int_\Sigma \left( \scal{ \g^\ggo \tr_\ggo \Arond }{ \g^\ggo H }_\ggo - \scal{ \di_\ggo \Arond }{ \g^\ggo H }_\ggo \right) \dd\vol_\ggo \nonumber \\[2mm]
			&=  \int_\Sigma \scal{ \tr_\ggo \g^\ggo \Arond - \di_\ggo \Arond }{ \g^\ggo H }_\ggo \ \dd\vol_\ggo . \label{eq:hard_term_scal}
		\end{align}
		We compute $\g^\ggo \Arond$:
		\begin{align}
			\g^\ggo_i \Arond_{kl} &= \dr_i \Arond_{kl} - (\Gamma^\ggo)_{ik}^p \Arond_{pl} - (\Gamma^\ggo)_{il}^p \Arond_{kp} \nonumber\\[2mm]
			&=  \dr_i \Arond_{kl} - \left[ \left( \Arond^{-1} \right)^{pb}\left( \g^g_i \Arond_{kb} \right) + (\g^\ggo H)^p \Arond_{ik} - \left( \Arond^{-1} \right)^p_{\ i} (\dr_k H) + (\Gamma^g)^p_{ik} \right] \Arond_{pl} \nonumber\\[2mm]
			&\quad - \left[ \left( \Arond^{-1} \right)^{pb}\left( \g^g_i \Arond_{lb} \right) + (\g^\ggo H)^p \Arond_{il} - \left( \Arond^{-1} \right)^p_{\ i} (\dr_l H) + (\Gamma^g)^p_{il} \right] \Arond_{kp} \nonumber\\[2mm]
			&= \g^g_i \Arond_{kl} - 2\g^g_i \Arond_{kl} - \Arond_{pl} (\g^\ggo H)^p \Arond_{ik} - \Arond_{kp} (\g^\ggo H)^p \Arond_{il} + g_{il} (\dr_k H) + g_{ik} (\dr_l H) \nonumber\\[2mm]
			&= -\g^g_i \Arond_{kl} - (\Arond^{-1})^{\alpha\beta} g_{l\alpha} (\dr_\beta H) \Arond_{ik} - (\Arond^{-1})^{\alpha\beta} g_{k\alpha} (\dr_\beta H) \Arond_{il} + g_{il} (\dr_k H) + g_{ik} (\dr_l H) . \label{eq:g_ggo_Arond}
		\end{align}
		We trace the indices $(k,l)$ in \eqref{eq:g_ggo_Arond} with respect to the metric $\ggo$:
		\begin{align}
			\tr_\ggo \g^\ggo_i \Arond &= (\ggo^{-1})^{kl} \g^\ggo_i \Arond_{kl} \nonumber\\[2mm]
			&= \left(\Arond^{-2} \right)^{kl} \left[ -\g^g_i \Arond_{kl} - \left( \Arond^{-1} \right)^{\alpha\beta} g_{l\alpha} (\dr_\beta H) \Arond_{ik} - \left( \Arond^{-1} \right)^{\alpha\beta} g_{k\alpha} (\dr_\beta H) \Arond_{il} + g_{il} (\dr_k H) + g_{ik} (\dr_l H)  \right] \nonumber\\[2mm]
			&= \g^g_i \left( \tr_g \Arond^{-1}\right) -2 \left( \Arond ^{-2} \right)_i ^{\ \beta} (\dr_\beta H) + \left(\Arond^{-2} \right)_i^{\ k} (\dr_k H) + \left(\Arond^{-2} \right)_i^{\ l} (\dr_l H) \nonumber\\[2mm]
			&= \g^g_i \left( \tr_g \Arond^{-1}\right) . \label{trace_gradient_arond}
		\end{align}
		We trace the indices $(i,k)$ in \eqref{eq:g_ggo_Arond} with respect to the metric $\ggo$:
		\begin{align*}
			(\di_\ggo \Arond)_l &= (\ggo^{-1})^{ik} \g^\ggo_i \Arond_{kl}\\[2mm]
			&= \left( \Arond^{-2} \right)^{ik} \left[  -\g^g_i \Arond_{kl} - \left( \Arond^{-1} \right)^{\alpha\beta} g_{l\alpha} (\dr_\beta H) \Arond_{ik} - \left( \Arond^{-1} \right)^{\alpha\beta} g_{k\alpha} (\dr_\beta H) \Arond_{il} + g_{il} (\dr_k H) + g_{ik} (\dr_l H)  \right] \\[2mm]
			&= - \left( \Arond^{-2} \right)^{ik} (\g^g_i \Arond_{kl}) - \left( \tr_g\Arond^{-1} \right) \left( \Arond^{-1} \right)^{\alpha\beta} g_{l\alpha} (\dr_\beta H) - \left(\Arond^{-2} \right)_l ^{\ \beta} (\dr_\beta H) + \left(\Arond^{-2} \right)_l ^{\ k}(\dr_k H) + \left(\tr_g \Arond^{-2} \right) (\dr_l H).
		\end{align*}
		We treat the first term using the Gauss--Codazzi identity:
		\begin{align*}
			\g^g_i \Arond_{kl} &= \g^g_l \Arond_{ik} + (\dr_l H)g_{ik} - (\dr_i H)g_{kl}.
		\end{align*}
		It follows that 
		\begin{align*}
			- \left( \Arond^{-2} \right)^{ik} (\g^g_i \Arond_{kl}) &= -\left( \Arond^{-2} \right)^{ik} \left( \g^g_l \Arond_{ik} + (\dr_l H)g_{ik} - (\dr_i H)g_{kl} \right) \\[2mm]
			&= \g^g_l \left( \tr_g \Arond^{-1} \right) - \left(\tr_g \Arond^{-2} \right) (\dr_l H) + \left(\Arond^{-2} \right)^i_{\ l} (\dr_i H).
		\end{align*}
		Hence, $\di_\ggo\Arond$ is given by
		\begin{align}
			(\di_\ggo \Arond)_l &= \g^g_l \left( \tr_g \Arond^{-1} \right)- \left( \tr_g\Arond^{-1} \right) \left( \Arond^{-1} \right)^{\alpha\beta} g_{l\alpha} (\dr_\beta H)+ \left(\Arond^{-2} \right)_l ^{\ k}(\dr_k H). \label{divergence_arond}
		\end{align}
		Returning to \eqref{eq:hard_term_scal}, we obtain
		\begin{align}
			\scal{ \tr_\ggo \g^\ggo \Arond - \di_\ggo \Arond }{ \g^\ggo H }_\ggo &= (\ggo^{-1})^{lp} \left[ \left( \tr_g\Arond^{-1} \right) \left( \Arond^{-1} \right)^{\alpha\beta} g_{l\alpha} (\dr_\beta H)- \left(\Arond^{-2} \right)_l ^{\ k}(\dr_k H) \right] (\dr_p H) \nonumber\\[2mm]
			&= \left[  \left( \tr_g\Arond^{-1} \right) \left( \Arond^{-3} \right)^{p\beta} (\dr_\beta H) - \left(\Arond^{-4} \right)^{kp} (\dr_k H) \right] (\dr_p H) \nonumber\\[2mm]
			&= \left[ \left( \tr_g\Arond^{-1} \right) \left( \Arond^{-3} \right)^{kp} - \left(\Arond^{-4} \right) ^{kp} \right] (\dr_k H) (\dr_p H). \label{eq:hard_term_scal1}
		\end{align}
		By Lemma \ref{calcul_a_inv}, $ \tr_g\Arond^{-1} = \frac{\tr_g\Arond^3}{3(\det\Arond)} $. Multiplying the Cayley--Hamilton identity \eqref{cayley_hamilton} by $\Arond^{-4}$ gives
		\begin{align*}
			g - \frac{|\Arond|^2_g }{2} \Arond ^{-2} - \frac{\tr_g\Arond^3}{3} \Arond^{-3} + (\det\Arond) \Arond ^{-4} &= 0.
		\end{align*}
		Hence, we obtain
		\begin{align}
			\left( \tr_g\Arond^{-1} \right) \left( \Arond^{-3} \right)^{kp} - \left(\Arond^{-4} \right) ^{kp} &= \left(\det\Arond \right)^{-1} \left[ g ^{kp} - \frac{|\Arond|^2}{2} \left(\Arond^{-2} \right)^{kp} \right]. \label{simplification_cayley_hamilton}
		\end{align}
		Substituting this into \eqref{eq:hard_term_scal1}, we obtain
		\begin{align*}
			\scal{ \tr_\ggo \g^\ggo \Arond - \di_\ggo \Arond }{ \g^\ggo H }_\ggo &=\left(\det\Arond \right)^{-1} \left[ g ^{kp} - \frac{|\Arond|^2}{2} \left(\Arond^{-2} \right)^{kp} \right](\dr_k H) (\dr_p H) \\[2mm]
			&= \left(\det\Arond \right)^{-1}\left[ |\g^g H|^2_g - \frac{|\Arond|^2_g}{2} |\g^\ggo H|^2_\ggo \right].
		\end{align*}
		We now return to \eqref{eq:integral_scal1}--\eqref{eq:integral_scal2}. Using $\dd\vol_\ggo = |\det_g \Arond|\, \dd\vol_g$ and $\ve\coloneqq \sign\left[\det_g\Arond\right]$, we obtain
		\begin{align*}
			\int_\Sigma \left(|B|^2_{\ggo,\eta} - 16\,|\vec{b}|^2_\eta\right) \dd\vol_\ggo &= 2\ve\int_\Sigma \left( |\g^g H|^2_g - \frac{|\Arond|^2_g}{2} |\g^\ggo H|^2_\ggo \right) \dd\vol_g \\[2mm]
			&- 6\int_\Sigma (\tr_\ggo \Arond) H\ \dd\vol_\ggo + \int_\Sigma f\left( |\Arond|^2_\ggo - (\tr_\ggo \Arond)^2  \right) \dd\vol_\ggo.
		\end{align*}
		By Lemma \ref{calcul_a_inv}, we obtain
		\begin{align*}
			\int_\Sigma \left(|B|^2_{\ggo,\eta} - 16\,|\vec{b}|^2_\eta\right) \dd\vol_\ggo &= 2\ve\int_\Sigma  \left( |\g^g H|^2_g - \frac{|\Arond|^2_g}{2} |\g^\ggo H|^2_\ggo\right)\dd\vol_g \\[2mm]
			&\quad - 2\ve\int_\Sigma (\tr_g \Arond^3) H\ \dd\vol_g + \ve\int_\Sigma f|\Arond|^2_g\, \dd\vol_g \\[2mm]
			&= \ve\int_\Sigma \left( 2|\g^g H|^2_g + H^2|\Arond|^2_g - 2H\left(\tr_g\Arond^3 \right)\right)  \dd\vol_g.
		\end{align*}
	\end{proof}
	
	We now compare the energies $\Egr$ and $\Sr$.
	\begin{proof}[Proof of Theorem \ref{thm_duality}]
		By Lemma \ref{lm:integral_ScalY}, if $\ve \coloneqq \sign\left[\det_g\Arond\right]$, then
		\begin{align*}
			\int_\Sigma \Scal_\ggo\, \dd\vol_\ggo &= 12\, \vol_\ggo(\Sigma) - \ve\int_\Sigma \left( 2|\g^g H|^2_g + H^2|\Arond|^2_g - 2H\left(\tr_g\Arond^3 \right)\right) \dd\vol_g \\[2mm]
			&=12\, \vol_\ggo(\Sigma) + \ve\int_\Sigma \left( -2|\g^g H|^2_g - H^2|\Arond|^2_g + 2H\left(\tr_g\Arond^3 \right)\right) \dd\vol_g .
		\end{align*}
		By \eqref{det_A}, we have
		\begin{equation*}
			\det A = H^4 - \frac{1}{2}H^2 |\Arond|^2_g + \frac{1}{3} H \left( \tr_g\Arond^3 \right) + \frac{1}{8} |\Arond|^4_g - \frac{1}{4} \left(\tr_g \Arond^4 \right).
		\end{equation*}
		Therefore, we obtain
		\begin{equation*}
			\int_\Sigma \Scal_\ggo\, \dd\vol_\ggo - 6\ve\int_\Sigma \det A\, \dd\vol_g = 12\, \vol_\ggo(\Sigma) + \ve\int_\Sigma  \left(-2|\g^g H|^2_g - 6H^4 + 2H^2 |\Arond|^2_g - \frac{3}{4} |\Arond|^4_g + \frac{3}{2} \left( \tr_g \Arond^4 \right)\right) \dd\vol_g.
		\end{equation*}
		By Lemma \ref{calcul_a_inv}, we have
		\begin{equation*}
			\frac{3}{2}\left( \tr_g \Arond^4 \right) = \frac{3}{4} |\Arond|^4_g - 6\det\Arond.
		\end{equation*}
		Using $\ve\, \det\Arond = |\det\Arond|$, we obtain 
		\begin{equation*}
			\int_\Sigma \Scal_\ggo\, \dd\vol_\ggo - 6\ve\int_\Sigma \det A\ \dd\vol_g = 6\,\vol_\ggo(\Sigma) -2\ve \,\Egr.
		\end{equation*}
		We conclude using \eqref{gauss_bonnet}.
	\end{proof}
	
	\subsection{The Paneitz operator}
	In this subsection, we prove the equality $\Egr=4\pi^2\chi(\Sigma) + \pr$ in Theorem \ref{thm_duality}, without the assumption \eqref{hyp:no_umbilic}, see Proposition \ref{pr:equality_paneitz} below. In view of \eqref{eq:integral_paneitz_general_formula}, we begin by computing $|\lap_g Y|^2_\eta$.

	\begin{lemma}\label{lm:laplacian_Y_eta}
		The Laplacian of $Y$ satisfies
		\begin{align*}
			|\lap_g Y|^2_\eta &= 4|\g^g H|^2_g + |\Arond|^4_g.
		\end{align*}
	\end{lemma}

	\begin{proof}
		Let $x_0\in \Sigma$. We compute $|\lap_g Y(x_0)|^2_\eta$ in normal coordinates for $g$ around $x_0$. We have
		\begin{align*}
			\lap_g Y &= \di_g\left( (\g^g H)\nu - \Arond \g \nu \right)\\[2mm]
			&= (\lap_g H)\nu + (\g^g H)(\g^g \nu) - (\di_g \Arond)(\g^g \nu) - \Arond \cdot [(\g^g)^2 \nu].
		\end{align*}
		By the Gauss--Codazzi identity, $\g^i A_{ij} = \g_j A_i^{\ i}$. Therefore, $\di_g \Arond + \g^g H = 4\g^g H$, and hence $\di_g \Arond = 3\g^g H$. Thus,
		\begin{equation}\label{laplace_Y}
			\lap_g Y =(\lap_g H)\nu -2 (\g^g H)(\g^g \nu) - \Arond \cdot [(\g^g)^2 \nu].
		\end{equation}
		Hence, its squared $\eta$-norm is given by
		\begin{align*}
			|\lap_g Y|^2_\eta =& -2(\lap_g H)\Arond^{ij} \scal{\nu}{\dr^2_{ij}\nu }_\eta \\[2mm]
			&\quad+ 4|\g^g H|^2_g +4\scal{(\g^g H)(\g^g \nu) }{\Arond \cdot [(\g^g)^2 \nu] }_\eta \\[2mm]
			&\quad + | \Arond \cdot [(\g^g)^2 \nu] |^2_\eta.
		\end{align*}
		Since $\scal{\nu}{\dr^2_{ij}\nu }_\eta = -\scal{\dr_i\nu}{\dr_j\nu}_\eta = -g_{ij}$, the first term vanishes. Since we work in normal coordinates, $\g^2 \Phi(x_0) \perp \g \Phi(x_0)$, and we have
		\begin{align*}
			\scal{\dr_i \nu}{\dr^2_{jk}\nu}_\eta &= \scal{\dr_i \Phi}{\dr^2_{jk}\Phi}_\xi + \scal{\dr_i\Phi}{\Phi}_\xi \dr_j \left( \scal{\dr_k\Phi}{\Phi}_\xi \right) - \scal{\dr_i\Phi}{\Phi}_\xi \dr_j \left( \scal{\dr_k\Phi}{\Phi}_\xi \right)\\[2mm]
			&= \scal{\dr_i \Phi}{\dr^2_{jk}\Phi}_\xi = \Gamma^g_{jki} = 0.
		\end{align*}
		Therefore, we obtain
		\begin{equation*}
			|\lap_g Y|^2_\eta =  4|\g^g H|^2_g + | \Arond \cdot [(\g^g)^2 \nu] |^2_\eta.
		\end{equation*}
		For the last term, using $\dr^2_{ij}\Phi(x_0)=\scal{\dr^2_{ij}\Phi(x_0)}{n(x_0)}_\xi n(x_0)$, we have
		\begin{equation*}
			| \Arond \cdot [(\g^g)^2 \nu] |^2_\eta = \Arond^{ij} \Arond^{\alpha\beta} \scal{ \begin{pmatrix} 
					\dr^2_{ij} \Phi \\ \dr_i\left(\scal{\dr_j\Phi}{\Phi}_\xi \right) \\ \dr_i\left( \scal{\dr_j\Phi}{\Phi}_\xi \right)
			\end{pmatrix} }{ \begin{pmatrix} 
					\dr^2_{\alpha\beta } \Phi \\ \dr_\alpha\left( \scal{\dr_\beta\Phi}{\Phi}_\xi \right) \\ \dr_\alpha\left( \scal{\dr_\beta\Phi}{\Phi}_\xi \right)
			\end{pmatrix} }_\xi = \Arond^{ij} \Arond^{\alpha\beta} A_{ij}A_{\alpha\beta} = |\Arond|^4_g. \qedhere
		\end{equation*}
	\end{proof}
	
	We now compute the main term $\Er_P \coloneqq \int_\Sigma \scal{Y}{P_gY}_\eta \dd\vol_g$.
	\begin{lemma}\label{lm:integral_Paneitz}
		We have
		\begin{equation*}
			\Er_P = \int_\Sigma \scal{Y}{P_g Y}_\eta \dd\vol_g =\int_\Sigma \left( 4|\g^g H|^2 + \frac{1}{3}|\Arond|^4_g + 2H^2 |\Arond|^2_g -4H\tr_g(\Arond^3)  + 2\tr_g(\Arond^4)\right) \dd\vol_g.
		\end{equation*}
	\end{lemma}
	
	\begin{proof}
		Using Lemma \ref{ricci_Phi} and Lemma \ref{scal_Phi}, we compute
		\begin{align*}
			& |\lap_g Y|^2 +\frac{2\Scal_g}{3} |\g^g Y|^2 - 2\Ric_g(\g^g Y,\g^g Y) \\[2mm]
			&= 4|\g^g H|^2 + |\Arond|^4_g + \frac{2}{3}(12H^2 - |\Arond|^2_g) |\Arond|^2_g - 2\left( 4H A_{\alpha\beta} - A_\alpha^\gamma A_{\gamma\beta}\right) \scal{\g^\alpha Y}{\g^\beta Y}_\eta  \\[2mm]
			& = 4|\g^g H|^2 + \frac{1}{3}|\Arond|^4_g + 8H^2 |\Arond|^2_g - 2(4 H A_{\alpha\beta} - (A^2)_{\alpha\beta} ) (\Arond^2)^{\alpha\beta}. 
		\end{align*}
		We compute the last term:
		\begin{align*}
			(4 H A_{\alpha\beta} - (A^2)_{\alpha\beta} ) (\Arond^2)^{\alpha\beta} &= 4H(\Arond_{\alpha\beta} + Hg_{\alpha\beta} ) (\Arond^2)^{\alpha\beta} - (\Arond^2 + 2H\Arond + H^2 g)_{\alpha\beta} (\Arond^2)^{\alpha\beta} \\[2mm]
			&= 4H\tr_g(\Arond^3) + 4H^2 |\Arond|^2_g - \tr_g(\Arond^4) - 2H\tr_g(\Arond^3)-H^2 |\Arond|^2_g \\[2mm]
			&= 2H\tr_g(\Arond^3) + 3H^2|\Arond|^2_g - \tr_g(\Arond^4).
		\end{align*}
		Therefore, we have
		\begin{align*}
			& |\lap_g Y|^2 +\frac{2\Scal_g}{3} |\g^g Y|^2 - 2\Ric_g(\g^g Y,\g^g Y) \nonumber\\[2mm]
			&= 4|\g^g H|^2 + \frac{1}{3}|\Arond|^4_g + 8H^2 |\Arond|^2_g -4H\tr_g(\Arond^3) - 6H^2|\Arond|^2_g + 2\tr_g(\Arond^4) \nonumber \\[2mm]
			& = 4|\g^g H|^2 + \frac{1}{3}|\Arond|^4_g +2H^2 |\Arond|^2_g -4H\tr_g(\Arond^3)  + 2\tr_g(\Arond^4).
		\end{align*}
	\end{proof}
	
	Using the Cayley--Hamilton theorem, we conclude as follows.
	
	\begin{proposition}\label{pr:equality_paneitz}
		Let $\Phi \colon \Sigma^4 \to \R^5$ be a smooth immersion. Its conformal Gauss map $Y \colon \Sigma \to \s^{5,1}$ satisfies
		\begin{equation*}
			\int_\Sigma \left( \scal{Y}{P_g Y}_\eta -\frac{4}{3} |\g Y|^4_{g,\eta} - 4\det_g(\Arond) \right) \dd\vol_g = 4\Egr - 16\pi^2 \chi(\Sigma).
		\end{equation*}
	\end{proposition}
	
	\begin{proof}
		By \eqref{det_A}, we have
		\begin{align*}
			\det A &= H^4 - \frac{1}{2}H^2 |\Arond|^2_g + \frac{1}{3} H \left( \tr_g\Arond^3 \right) + \frac{1}{8} |\Arond|^4_g - \frac{1}{4} \left(\tr_g \Arond^4 \right).
		\end{align*}
		We obtain the following expression for $H\tr_g(\Arond^3)$:
		\begin{align*}
			-4 H\tr_g(\Arond^3) &= -12\det_g(A) +12H^4 - 6H^2|\Arond|^2_g + \frac{3}{2} |\Arond|^4_g - 3\tr_g(\Arond^4).
		\end{align*}
		Using this together with Lemma \ref{lm:integral_Paneitz}, we obtain
		\begin{align}
			\int_\Sigma \scal{Y}{P_g Y}_\eta d\vol_g &= \int_\Sigma 4|\g^g H|^2 + \frac{1}{3}|\Arond|^4_g + 2H^2 |\Arond|^2_g -4H\tr_g(\Arond^3)  + 2\tr_g(\Arond^4)\ d\vol_g \nonumber \\[2mm]
			&= \int_\Sigma 4|\g^g H|^2 + 12H^4 - 4H^2 |\Arond|^2_g + \frac{11}{6}|\Arond|^4_g - \tr_g(\Arond^4)- 12 \det(A)\ d\vol_g \nonumber\\[2mm]
			&= 4\Egr - 16\pi^2 \chi(\Sigma) + \int_\Sigma \frac{11}{6}|\Arond|^4_g - \tr_g(\Arond^4)\ d\vol_g. \label{eq:comparison_Paneitz_W}
		\end{align}
		Since $8\det_g(\Arond) = |\Arond|^4_g - 2\tr_g(\Arond^4)$ by \eqref{det_Arond}, and $|\g Y|^2_g = |\Arond|^2_g$, we have
		\begin{equation*}
			4 \det_g(\Arond) + \frac{4}{3}|\g Y|^4_{g,\eta} = \frac{11}{6}|\Arond|^4_g - \tr_g(\Arond^4).\qedhere
		\end{equation*}
	\end{proof}

	\subsection{Variational interpretation}\label{subsec:variational_interpretation}
	
	Theorem \ref{thm_duality} implies a constrained variational correspondence in the following sense. An immersion \(\Phi\) is a critical point of \(\Egr\) if and only if the first variation of \(\pr\) vanishes along every family arising from a variation of \(\Phi\). Under the nondegeneracy assumption \(\det_g\Arond\neq0\), the same is true of \(\Sr\). Thus, the admissible variations of \(Y\) are restricted to families of conformal Gauss maps associated with immersions into \(\R^5\).
	
	To derive the corresponding Euler--Lagrange equations, one must first characterise the admissible variation fields $\left.\partial_tY_t\right|_{t=0}$, where \(Y_t\) is the conformal Gauss map of a smooth family of immersions \(\Phi_t\colon\Sigma\to\R^5\). In the case of \(\pr\), one must also account for the variation of the metric \(g_t=\Phi_t^*\xi\). Its contribution is essential when the resulting Euler--Lagrange equation is expressed in terms of \(Y\). A similar characterisation of these admissible variations for surfaces in \(\R^3\) was carried out in \cite{M2025}.
	
	In the present four-dimensional setting, the corresponding calculations are given in \cite[Sections~4.8 and~4.9]{martino2024}. The resulting system is considerably more involved than its two-dimensional counterpart because, as shown in Proposition~\ref{if_Y_conformal}, the conformal Gauss map is not a conformal map in general. For example, the Euler--Lagrange equation associated with \(\pr\) contains the \(4\)-Laplacian $\di_g\bigl(|\g Y|_{g,\eta}^2\,\g^gY\bigr)$, the Paneitz operator \(P_gY\), the Laplacian of this term, and several lower-order contributions.
	
	The leading-order part of the Euler--Lagrange system associated with \(\Sr\) contains the contraction of the Einstein tensor of the induced metric \(\ggo=Y^*\eta\) with the second fundamental form of \(Y\). This structure arises because \(\Sr\) is varied through metrics induced by admissible families of immersions. The variations under consideration therefore do not range over all Riemannian metrics on \(\Sigma\), but only over the constrained class induced by conformal Gauss maps of immersions into \(\R^5\).
	
	\section{Conformal induced metrics and totally umbilic foliations}\label{se:no_conformal}
	
	In this section, we prove Proposition \ref{if_Y_conformal}. Let $\Phi \colon \B^4 \to \R^5$ be a smooth immersion, where $\B^4$ is the unit ball of $\R^4$ centred at the origin. Assume that there exists $\omega \in C^\infty(\B^4;(0,+\infty))$ such that $\Arond^2 = \omega^2 g$. In Lemma \ref{lm:eigenvalues_Arond}, we show that the eigenspaces of $\Arond$ define an orthogonal splitting of the tangent space into two two-dimensional distributions. In Lemma \ref{lm:frobenius_eigenspaces}, we show that these eigenspaces are involutive. By the Frobenius theorem, we obtain coordinates that split $\Phi(\B^4)$ into two transverse foliations by surfaces. In Lemma \ref{lm:round_leaves}, we prove that the restriction of the metric $g$ to each of these surfaces is a round metric on a two-dimensional manifold. 
	
	\subsection{Eigenvalue structure of $\Arond$}
	
	\begin{lemma}\label{lm:eigenvalues_Arond}
		At any $p\in \B^4$, $\Arond(p)$ has exactly two eigenvalues, $-\omega(p)$ and $\omega(p)$, each with multiplicity two. 
	\end{lemma}
	\begin{proof}[Proof]
		Let $p\in \B^4$. Since $\Arond$ is symmetric, we can find an orthonormal basis $e_1,\ldots,e_4$ of $\dd\Phi_p(T_p\B^4)$ such that, for each $i\in\inter{1}{4}$, $\Arond e_i = \lambda_i e_i$. Applying $\Arond$ again gives $\lambda_i^2 = \omega^2$. Since $\tr_g \Arond = 0$, we have $\lambda_4 = -\lambda_1 - \lambda_2 - \lambda_3$. Since $\lambda_4 = \pm \omega$, two of the $(\lambda_i)_{1\leq i\leq 3}$ must have the same sign, while the third has the opposite sign. After relabelling the $\lambda_i$ for $i\in\inter{1}{3}$, we may assume that $\lambda_1=\lambda_2=\omega$ and $\lambda_3=-\omega$. Thus, $\lambda_4 = -\omega$. Therefore, in the basis $(e_1,e_2,e_3,e_4)$, we have
		\begin{equation}\label{eq:Arond_diag}
			\Arond(p) = \omega \begin{pmatrix}
				1 & 0 & 0 & 0\\
				0 & 1 & 0 & 0\\
				0 & 0 & -1 &0 \\
				0 & 0 & 0 & -1
			\end{pmatrix}.
		\end{equation}
	\end{proof}
	Since the two eigenvalues $\pm \omega$ are distinct throughout $\B^4$ and have fixed multiplicities, the bundles $\ker\left( \Arond\pm \omega g \right)$ are smooth vector bundles of dimension 2, see for instance \cite[Chapter I, Section 5, or Chapter II, Section 5]{kato1995}. We therefore have the orthogonal decomposition
	\begin{align}\label{eq:orthogonal_decompo_TPhi}
		\dd\Phi_p(T_p\B^4) = \ker\left( \Arond-\omega g \right)_{p} \oplus \ker\left(\Arond+\omega g\right)_{p}.
	\end{align}
	
	\subsection{Integrability of the eigendistributions}
	
	\begin{lemma}\label{lm:frobenius_eigenspaces}
		The two distributions $\ker\left(\Arond\pm \omega g\right)$ are involutive in the following sense. Given any vector fields $x,y \in \ker\left(\Arond -\omega g\right)$, we have $[x,y]\in \ker\left(\Arond - \omega g\right)$. The same statement holds with $\ker\left(\Arond -\omega g\right)$ replaced by $\ker\left(\Arond +\omega g\right)$.
	\end{lemma}
	\begin{proof}[Proof]
		Since the decomposition \eqref{eq:orthogonal_decompo_TPhi} is smooth, there exists a smooth orthonormal frame $(e_1,e_2,e_3,e_4)$ such that $(e_1,e_2)$ is a frame for $\ker(\Arond- \omega g)$ and $(e_3,e_4)$ is a frame for $\ker(\Arond+\omega g)$. 
		
		Fix $p \in \B^4$. Consider normal coordinates $(t^1 ,\ldots, t^4)$ around $p$ satisfying the following condition at the origin:
		\begin{equation*}
			\forall i\in\inter{1}{4},\qquad \left.\frac{\dr}{\dr t^i}\right|_{t=0} = e_i(p).
		\end{equation*} 
		Let $X,Y\in \ker(\Arond - \omega g)_p$. We show that $[X,Y]\in \ker(\Arond - \omega g)_p$. We decompose
		\begin{equation*}
			X =  x^1e_1 + x^2 e_2,\qquad \text{ and }\qquad  Y = y^1 e_1+y^2 e_2.
		\end{equation*}
		At the origin, we have
		\begin{align}
			\Arond[X,Y] &= \sum_{1\leq i,j\leq 2}\Arond \left[ x^i e_i, y^j e_j\right]  \nonumber \\
			&= \sum_{1\leq i,j\leq 2} \Arond\left( x^i e_i(y^j) e_j - y^j e_j(x^i) e_i + x^i y^j[e_i,e_j]  \right) \nonumber \\
			&= \sum_{1\leq i,j\leq 2}\left( \omega x^i e_i(y^j) e_j -\omega y^j e_j(x^i) e_i + x^i y^j \Arond[e_i,e_j]  \right). \label{eq:Arond_bracket}
		\end{align}
		We compute the last term using the fact that, in the chosen normal coordinates, the Christoffel symbols vanish at the origin:
		\begin{align}
			\Arond[e_1,e_2] & = \Arond \left( \g_{e_1} e_2 - \g_{e_2} e_1 \right) \nonumber \\
			&= \Arond \left( \g_{\frac{\dr}{\dr t^1}} e_2 - \g_{\frac{\dr}{\dr t^2}} e_1 \right) \nonumber\\
			&= \g_{\frac{\dr}{\dr t^1}}\left( \Arond e_2 \right) - \g_{\frac{\dr}{\dr t^2}} \left( \Arond e_1 \right) - \left( \g_{\frac{\dr}{\dr t^1}}\Arond \right) e_2 + \left( \g_{\frac{\dr}{\dr t^2}} \Arond \right) e_1 \nonumber\\
			&= \g_{\frac{\dr}{\dr t^1}}\left( \Arond e_2 \right) - \g_{\frac{\dr}{\dr t^2}} \left( \Arond e_1 \right) - \left( \g_{\frac{\dr}{\dr t^1}}\Arond \right) \frac{\dr}{\dr t^2} + \left( \g_{\frac{\dr}{\dr t^2}} \Arond \right) \frac{\dr}{\dr t^1}. \label{eq:Arond_bracket2}
		\end{align}
		The first two terms satisfy
		\begin{align*}
			\g_{\frac{\dr}{\dr t^1}}\left( \Arond e_2 \right) - \g_{\frac{\dr}{\dr t^2}} \left( \Arond e_1 \right) &= \g_{\frac{\dr}{\dr t^1}}\left(  \omega e_2 \right) - \g_{\frac{\dr}{\dr t^2}} \left( \omega e_1 \right) \\
			&= e_1(\omega) e_2 - e_2(\omega) e_1 + \omega \left( \g_{e_1} e_2 - \g_{e_2} e_1 \right) \\
			&= e_1(\omega) e_2 - e_2(\omega) e_1 + \omega[e_1,e_2].
		\end{align*}
		We denote by $A^{*i}$ the $i$-th column of the matrix $A$. The last two terms in \eqref{eq:Arond_bracket2} satisfy
		\begin{align*}
			- \left( \g_{\frac{\dr}{\dr t^1}}\Arond \right) \frac{\dr}{\dr t^2} + \left( \g_{\frac{\dr}{\dr t^2}} \Arond \right) \frac{\dr}{\dr t^1} &= \g_2 \Arond^{*1} - \g_1\Arond^{*2} \\
			&= \g_2 (A^{*1} - Hg^{*1}) - \g_1(A^{*2}  - Hg^{*2}).
		\end{align*}
		By the Gauss--Codazzi equations, we have
		\begin{align*}
			- \left( \g_{\frac{\dr}{\dr t^1}}\Arond \right) \frac{\dr}{\dr t^2} + \left( \g_{\frac{\dr}{\dr t^2}} \Arond \right) \frac{\dr}{\dr t^1} &= \g A^{21} - (\g_2 H)g^{*1} - \g A^{12} + (\g_1 H) g^{*2} \\
			&=-(\g_2 H) e_1 + (\g_1 H) e_2.
		\end{align*}
		Returning to \eqref{eq:Arond_bracket2}, we deduce that
		\begin{align}\label{eq:Arond_bracket3}
			\Arond[e_1,e_2] = \omega[e_1,e_2] + \g_1(\omega+H) e_2 - \g_2(\omega+H) e_1.
		\end{align}
		We claim that $[e_1,e_2]\in \Span(e_1,e_2)$. To prove this claim, we take the scalar product with $e_i$, for $i\in\{3,4\}$, and use the preceding identity:
		\begin{align*}
			-\omega \scal{[e_1,e_2]}{e_i} = \scal{[e_1,e_2]}{\Arond e_i} = \scal{\Arond[e_1,e_2]}{e_i} = \omega\scal{[e_1,e_2]}{e_i}.
		\end{align*}
		Hence, we obtain
		\begin{align*}
			2\omega\scal{[e_1,e_2]}{e_i} = 0.
		\end{align*}
		Consequently, $[e_1,e_2] \in \Span(e_3,e_4)^\perp = \Span(e_1,e_2)= \ker\left( \Arond - \omega g \right)$. Thus, relation \eqref{eq:Arond_bracket3} reduces to $\g_2 (\omega+ H)= 0$, $\g_1 (\omega+H)=0$, and
		\begin{align*}
			\Arond [e_1,e_2] = \omega [e_1,e_2].
		\end{align*}
		Returning to \eqref{eq:Arond_bracket}, we deduce that
		\begin{align*}
			\Arond [X,Y] = \omega[X,Y].
		\end{align*}
		Thus, the distribution $\ker\left( \Arond - \omega g\right)$ is involutive. The same argument shows that the distribution $\ker\left( \Arond + \omega g\right)$ is also involutive.
	\end{proof}
	
	\subsection{Totally umbilic leaves}
	
	By the Frobenius theorem, see for instance \cite[Theorem 1.60]{warner1983}, there exist coordinates $(t^1,\ldots,t^4)$ around any point of $\B^4$ such that $\left(\frac{\dr}{\dr t^1},\frac{\dr}{\dr t^2} \right)$ is a basis for $\ker\left( \Arond-\omega g\right)$ and $\left( \frac{\dr}{\dr t^3},\frac{\dr}{\dr t^4}\right)$ is a basis for $\ker\left(\Arond +\omega g\right)$. In particular, for any constant $c\in\R^4$, the equations $\{t^3=c^3,t^4=c^4\}$ and $\{t^1=c^1,t^2=c^2\}$ define integral manifolds of the distributions $\ker\left(\Arond -\omega g\right)$ and $\ker\left(\Arond +\omega g\right)$, respectively. We show that each immersed surface $\Phi(\{t^3=c^3,t^4=c^4\})$ is totally geodesic. 
	
	\begin{lemma}\label{lm:round_leaves}
		Let $c\in\R^4$ lie in the range of a coordinate system $(t^1,\ldots,t^4)$. Define $\Psi_+(t^1,t^2) \coloneqq  \Phi(t^1,t^2,c^3,c^4)$. Then the submanifold $\Ima(\Psi_+)$ is totally umbilic in $\R^5$. Hence, it is either a flat 2-plane or a round 2-sphere.
		
		Similarly, the map $\Psi_-(t^3,t^4) \coloneqq  \Phi(c^1,c^2,t^3,t^4)$ defines an immersion such that $\Ima(\Psi_-)$ is totally umbilic.
	\end{lemma}
	\begin{proof}[Proof]
		To simplify the notation, let $\Psi = \Psi_+$. Consider $q=\Psi(t^1_0,t^2_0)$ and choose normal coordinates $(\tau^1,\ldots,\tau^4)$ for $\Phi$ around $q$ such that $\Arond_\Phi$ has the form \eqref{eq:Arond_diag} at $q$ in the basis $(e_1(q),\ldots,e_4(q)) = \left( \dr_{\tau^1}\Phi(0) ,\ldots, \dr_{\tau^4}\Phi(0)\right)$. In particular, $\left( \dr_{\tau^1}\Phi(0), \dr_{\tau^2}\Phi(0) \right)$ is a basis for $\ker\left( \Arond - \omega g\right)_q$, and $\left( \dr_{\tau^3}\Phi(0), \dr_{\tau^4}\Phi(0) \right)$ is a basis for $\ker\left( \Arond + \omega g\right)_q$. Hence, there exist vectors $u_1,u_2\in\R^2$ such that the following relations hold:
		\begin{align*}
			& q= \Phi(0)= \Psi(t^1_0,t^2_0),\\
			& \dd\Psi(t^1_0,t^2_0)\cdot u_1 = \dr_{\tau^1} \Phi(0)=e_1(q),\\
			& \dd\Psi(t^1_0,t^2_0)\cdot u_2 = \dr_{\tau^2} \Phi(0)=e_2(q),\\
			& \dr_{\tau^3} \Phi(0)=e_3(q),\\
			& \dr_{\tau^4} \Phi(0)=e_4(q),\\
			& \Gamma_{ij}^k (0) = 0.
		\end{align*}
		By inverting the second and third relations, we obtain a matrix $M = (m_i^{\, j})_{1\leq i,j\leq 2}$ such that
		\begin{align*}
			\forall i\in\{1,2\},\qquad \dr_{t^i} \Psi(t_0^1,t_0^2) = m_i^{\, j}\, \dr_{\tau^j} \Phi(0).
		\end{align*}
		Hence, the normal space of $\Ima(\Psi)$ at $q$ is given by
		\begin{align*}
			(\dd\Psi_q(\R^2))^\perp = \Span(\dr_{\tau^3}\Phi(0),\dr_{\tau^4}\Phi(0),n(0)) = \Span(e_3(q),e_4(q),n(0)).
		\end{align*}
		The second fundamental form $A_\Psi^{\dr_k\Phi}(q)$ of $\Psi$ in the direction $\dr_k \Phi(0)$ is given by
		\begin{align}
			A_\Psi^{\dr_k\Phi}(q) & = \left( -\scal{\g_{\dr_i \Psi(t_0^1,t_0^2)} \dr_k \Phi(0) }{\dr_j \Psi(t_0^1,t_0^2)} \right)_{1\leq i,j\leq 2} \nonumber\\[2mm]
			&= \left( -m_i^{\, \alpha}\, m_j^{\, \beta}\, \scal{\g_{\dr_\alpha \Phi(0)} \dr_k \Phi(0) }{\dr_\beta \Phi(0)} \right)_{1\leq i,j\leq 2} \nonumber\\[2mm]
			&= \left( -m_i^{\, \alpha}\, m_j^{\, \beta}\, \Gamma_{\alpha k\, \beta}(0) \right)_{1\leq i,j\leq 2} = 0. \label{eq:A_Psi_tangent}
		\end{align}
		The second fundamental form $A_\Psi^{n}$ of $\Psi$ in the direction $n$ is given by
		\begin{equation*}
			A_{\Psi}^n(q) = \left( -\scal{ \g_{\dr_i \Phi(0)} n}{\dr_j \Phi(0)} \right)_{1\leq i,j\leq 2}= \left( -\scal{ \g_{e_i} n}{e_j} \right)_{1\leq i,j\leq 2}.
		\end{equation*}
		From \eqref{eq:Arond_diag}, we obtain
		\begin{equation*}
			A_{\Psi}^n(q) = \begin{pmatrix}
				H+\omega & 0\\
				0 & H+\omega
			\end{pmatrix}.
		\end{equation*}
		Hence, its traceless part satisfies $\Arond_{\Psi}^n(q) = 0$. Together with \eqref{eq:A_Psi_tangent}, this implies that
		\begin{equation*}
			\forall v\in (\dd\Psi_q(\R^2))^\perp,\qquad \Arond_\Psi^v(q) = 0.
		\end{equation*}
		Since $q$ is arbitrary, we deduce that $\Psi$ is totally umbilic.
	\end{proof}

	\section{Bounds for the two conformal energies}\label{se:energy_bounds}
	
	\subsection{Unboundedness of $\Egr$}\label{se:infimum}
	
	The proof of Theorem \ref{prop:min_W_spheres} proceeds as follows. First, we show that, among immersions of $\s^4$, the functional $\Egr$ is not bounded below by constructing a sequence $(\Phi_k)_{k\in\N}$ such that $\Egr(\Phi_k)\xrightarrow[k\to\infty]{}{-\infty}$, see Lemma \ref{lm:inf_W_spheres}. We then obtain Theorem \ref{prop:min_W_spheres} by taking the connected sum of each immersion in the sequence $(\Phi_k)_{k\in\N}$ with any closed hypersurface $\Sigma^4 \subset \R^5$. This operation preserves the topology while the energy $\Egr$ tends to $-\infty$.
	
	We first show that the integrand $|\g H|^2 - H^2|A|^2 + 7H^4$ is a negative constant for isometric immersions of $\R^2 \times\s^2$.
	
	\begin{lemma}\label{lm:II_R2_S2}
		Consider the following isometric immersion of $\R^2 \times \s^2$ into $\R^5$:
		\begin{align*}
			\forall ((x,y),p)\in\R^2\times\s^2,\qquad \iota((x,y),p) = \begin{pmatrix}
				x\\ y \\ p
			\end{pmatrix} \in \R^5.
		\end{align*} 
		We have
		\begin{align*}
			|\g H_{\iota}|^2_{g_\iota} - H^2_{\iota} |A_\iota|^2_{g_\iota} + 7H^4_{\iota} = -\frac{1}{16}.
		\end{align*}
	\end{lemma}	
	
	\begin{proof}
		The second fundamental form of $\iota$ is given by 
		\begin{align*}
			A_\iota = \begin{pmatrix}
				0 & 0 & 0 & 0\\
				0 & 0 & 0 & 0\\
				0 & 0 & 1 & 0 \\
				0 & 0 & 0 & 1
			\end{pmatrix}.
		\end{align*}
		Hence, $|A_\iota|^2 = 2$ and $H_\iota = \frac{1}{2}$. Thus,
		\begin{equation*}
			|\g H_{\iota}|^2_{g_\iota} - H^2_{\iota} |A_{\iota}|^2_{g_{\iota}} + 7H^4_{\iota} = -2 \frac{1}{2^2} + 7 \frac{1}{2^4} = \frac{-8 +7}{16}.\qedhere
		\end{equation*}
	\end{proof}
	
	By inserting a manifold of the form $[0,L]^2 \times \s^2$ into $\s^4$, we obtain the following result, from which Theorem \ref{prop:min_W_spheres} follows.
	
	\begin{lemma}\label{lm:inf_W_spheres}
		There exist $\delta\in(0,\frac{\pi}{4})$ and $L_0>1$ with the following property. For any $L>L_0$, there exists a smooth hypersurface $S_L\subset \R^5$ having the same topology as $\s^4$ and satisfying
		\begin{align*}
			\Egr(S_L) \leq -\delta L^2.
		\end{align*}
	\end{lemma}
	
	\begin{proof}
		Consider the round sphere $\s^4 \subset \R^5$, and denote by $x_1,\ldots,x_5$ the coordinates on $\R^5$. The equator of $\s^4$ is isometric to $\s^3$:
		\begin{align*}
			\{x\in \s^4 : x_1 = 0\} = \{ x\in\R^5 : x_1 =0, x_2^2 + x_3^2 + x_4^2 + x_5^2 = 1\}.
		\end{align*}
		We cut $\s^4$ along the equator to obtain $\s^4 \cap \{x_1<0\}$ and $\s^4 \cap \{x_1>0\}$. The boundaries of these two manifolds are isometric to $\s^3$. For a given $L>1$, we connect these two boundaries by a straight cylinder $[0,L]\times\s^3$. To obtain a smooth hypersurface, we only need to mollify the boundaries of $[0,L]\times\s^3$, $\s^4 \cap \{x_1<0\}$, and $\s^4 \cap \{x_1>0\}$. This requires modifying only a small part of each manifold near its boundary. Thus, we can glue these three pieces into a smooth manifold $S^{(1)}_L = S_+^{(1)}\cup G_L^{(1)} \cup C_L^{(1)} \cup S_-^{(1)}$ consisting of four pieces: 
		\begin{enumerate}
			\item $S_-^{(1)}$ is isometric to $\s^4 \cap \{x_1 < -\frac{1}{10}\}$,
			\item $S_+^{(1)}$ is isometric to $\s^4\cap \{x_1 > \frac{1}{10}\}$,
			\item $C_L^{(1)}$ is isometric to $[\frac{1}{10}, L - \frac{1}{10}] \times\s^3$,
			\item $G_L^{(1)}$ is the mollified part. 
		\end{enumerate}
		To estimate $\Egr(S_L^{(1)})$, we use the fact that the gluing region $G_L^{(1)}$ has volume independent of $L$. Thus, there exists a constant $\Gamma>1$ such that, for any $L>1$,
		\begin{align*}
			\Egr(S_L^{(1)}) \leq \Gamma + \int_{ C_L^{(1)} } \left( |\g H|^2 - H^2|A|^2 + 7H^4\right).
		\end{align*}
		The second fundamental form of $C_L^{(1)} \simeq [\frac{1}{10}, L - \frac{1}{10}] \times\s^3$ is given by
		\begin{align*}
			A_{\R\times\s^3} = \begin{pmatrix}
				0 & 0 & 0 & 0\\
				0 & 1 & 0 & 0\\
				0 & 0 & 1 & 0\\
				0 & 0 & 0 & 1
			\end{pmatrix}.
		\end{align*}
		Thus, after increasing $\Gamma$ if necessary, we obtain the estimate
		\begin{align}\label{eq:W_CL}
			\Egr(S_L^{(1)}) \leq \Gamma(1+L).
		\end{align}
		
		We proceed in the same manner in the $x_2$-direction. We cut $S_L^{(1)}$ along $S_L^{(1)} \cap \{x_2=0\}$ and consider the two remaining parts $S_L^{(1)}\cap \{x_2<0\}$ and $S_L^{(1)} \cap \{x_2>0\}$. We glue them along the boundary of $[0,L]\times (S_L^{(1)}\cap \{x_2=0\})$. Again, the region required for the gluing is a small neighbourhood of $S_L^{(1)}\cap \{x_2=0\}$. More precisely, we keep the parts $S^{(2)}_- \coloneqq S_L^{(1)}\cap \{x_2<-\frac{1}{10}\}$, $S^{(2)}_+ \coloneqq S_L^{(1)} \cap \{x_2>\frac{1}{10}\}$, and $C^{(2)}_L \coloneqq [\frac{1}{10},L-\frac{1}{10}]\times (S_L^{(1)}\cap \{x_2=0\})$ fixed, while the mollification needed to obtain a smooth manifold occurs only in the regions $C_L\cap \{0>x_2>-\frac{1}{10}\}$, $C_L \cap \{0<x_2<\frac{1}{10}\}$, and $\big( [0,\frac{1}{10}]\cup[L-\frac{1}{10},L] \big) \times (C_L\cap \{x_2=0\})$. We denote this mollified part by $G_L^{(2)}$. The product structure shows that the energy required for the gluing can be chosen to be uniformly bounded. Let $S_L^{(2)}\coloneqq S^{(2)}_- \cup S^{(2)}_+ \cup C_L^{(2)} \cup G_L^{(2)}$ be the resulting submanifold. To estimate $\Egr(S_L^{(2)})$, we consider three distinct parts.
		\begin{itemize}
			\item On the parts $S^{(2)}_- \cup S^{(2)}_+$, the energy can be estimated as in \eqref{eq:W_CL}, showing that it grows linearly in $L$.
			
			\item On the part $G_L^{(2)}$, the energy again grows linearly in $L$. Indeed, the gluing regions are $S_L^{(1)}\cap \{0>x_2>-\frac{1}{10}\}$, $S_L^{(1)} \cap \{0<x_2<\frac{1}{10}\}$, and $\big( [0,\frac{1}{10}]\cup[L-\frac{1}{10},L] \big) \times (S_L^{(1)}\cap \{x_2=0\})$. As explained above, the energy of the gluing can be bounded uniformly, and we obtain
			\begin{align*}
				\int_{G_L^{(2)}} \left(|\g H_{S_L^{(2)}}|^2 -H_{S_L^{(2)}}^2 |A_{S_L^{(2)}}|^2 + 7H_{S_L^{(2)}}^4 \right) & \leq   \Gamma \left[ \Big|S_L^{(1)}\cap \{0>x_2>-\frac{1}{10}\} \Big| + \Big| S_L^{(1)} \cap \{0<x_2<\frac{1}{10}\} \Big|\right. \\[2mm]
				&\quad \left. + \Big| \big( [0,\frac{1}{10}]\cup[L-\frac{1}{10},L] \big) \times (S_L^{(1)}\cap \{x_2=0\}) \Big| \right] \\[2mm]
				& \leq \Gamma\, (1+L).
			\end{align*}
			Here, we have used the fact that the diameter of $S_L^{(1)}$ in the $x_1$-direction is controlled by $L$ and is bounded in every other coordinate direction.
			
			\item On the part $C_L^{(2)}$, the integrand of $\Egr$ is computed differently in three cases. As above, the integrand $|\g H|^2 -H^2|A|^2+7H^4$ can be bounded uniformly on the parts involving $(S_+^{(1)}\cup S_-^{(1)})\cap \{x_2=0\}$ and $G_L^{(1)}\cap \{x_2=0\}$. Thus, after integration, the energy is again bounded by a linear function of $L$. On the part of $C_L^{(1)}\cap \{x_2=0\}$ that is isometric to $\left( [\frac{1}{10},L-\frac{1}{10}]\times \s^3\right) \cap \{x_2=0\} = [\frac{1}{10},L-\frac{1}{10}]\times \s^2 $, the integrand is given by Lemma \ref{lm:II_R2_S2}. Moreover, we have $[\frac{1}{10},L-\frac{1}{10}]\times \left( C_L^{(1)}\cap \{x_2=0\} \right) = [\frac{1}{10},L-\frac{1}{10}]^2 \times \s^2$, and hence
			\begin{align*}
				\int_{ [\frac{1}{10},L-\frac{1}{10}]^2\times \s^2 } \left( |\g H|^2 -H^2|A|^2+7H^4\right) = -\frac{\vol(\s^2)}{16}\left( L-\frac{1}{5} \right)^2.
			\end{align*}
		\end{itemize}
		Combining the estimates for all the pieces gives
		\begin{align*}
			\Egr(S_L^{(2)}) \leq \Gamma(1+L) - \frac{\pi}{4}L^2.
		\end{align*}
	\end{proof}
	
	\subsection{Positivity of $\Ep$}\label{se:inf_paneitz}	
	
	Proposition \ref{prop:min_P} follows from a straightforward application of the Cauchy--Schwarz inequality.
	
	Let $\Phi \colon \Sigma^4 \to \R^5$ be a smooth immersion, let $g$ be the metric induced by $\Phi$ on $\Sigma$, and let $Y$ be its conformal Gauss map. By Lemma \ref{lm:integral_Paneitz}, we have
	\begin{equation*}
		\Er_P =\int_\Sigma \left( 4|\g^g H|^2 + \frac{1}{3}|\Arond|^4_g + 2H^2 |\Arond|^2_g -4H\tr_g(\Arond^3)  + 2\tr_g(\Arond^4)\right) \dd\vol_g.
	\end{equation*}
	By the Cauchy--Schwarz inequality, we bound the term $H \tr_g(\Arond^3)$ as follows:
	\begin{equation*}
		2H\tr_g(\Arond^3) = 2\scal{H\Arond}{\Arond^2}_g \leq H^2|\Arond|^2_g + |\Arond^2|^2_g = H^2|\Arond|^2_g + \tr_g(\Arond^4).
	\end{equation*}
	By Lemma \ref{lm:integral_Paneitz}, we obtain
	\begin{equation}\label{eq:lower_bound_Paneitz}
		\Er_P \geq \int_\Sigma \left(4|\g^g H|^2_g + \frac{1}{3} |\Arond|^4_g \right) \dd\vol_g \geq 0.
	\end{equation}
	Equality holds in the last inequality if and only if $\g H = 0$ and $\Arond = 0$. Equivalently, $H$ is constant and $\Phi$ is totally umbilic.
	
	\appendix
	
	\section{Algebraic identities in dimension four}\label{sec:cayley_ham}
	
	We now obtain an expression in terms of the metric $g$. To do so, we need an expression for $\Arond^{-1}$.
	\begin{lemma}\label{calcul_a_inv}
		Let $\Phi \colon \Sigma^4\to\R^5$ be a smooth immersion satisfying \eqref{hyp:no_umbilic}. Its traceless second fundamental form $\Arond$ satisfies
		\begin{align}
			\tr_g\left( \Arond^4 \right) &= \frac{1}{2} \left[ \tr_g \left( \Arond^2 \right) \right]^2 - 4\det\left( \Arond \right), \label{det_Arond} \\
			\Arond^{-1} &= \frac{-1}{\det\left( \Arond \right) } \left[ \Arond^3 - \frac{1}{2} \tr_g \left( \Arond^2 \right) \Arond - \frac{1}{3} \tr_g\left( \Arond^3 \right) g \right], \nonumber\\
			\left| \Arond^{-1} \right|^2_g &= \det\left( \Arond \right)^{-2} \Bigg[  \left( \tr_g \Arond^2 \right) \left(\det \Arond \right) + \frac{1}{9} \left( \tr_g \Arond^3 \right)^2 \Bigg], \nonumber\\
			\tr_g \left( \Arond^{-1} \right) &= \frac{1}{3\det\left( \Arond \right) } \tr_g\left( \Arond^3\right). \label{eq:tr_Arond_inv}
		\end{align}
	\end{lemma}
	
	\begin{proof}
		We apply the Cayley--Hamilton theorem to the matrix $\Arond g^{-1}$:
		\begin{equation*}
			\Arond^4 - \left( \tr_g \Arond \right) \Arond^3 + \frac{1}{2}\left[ \left( \tr_g \Arond \right)^2 - \tr_g \left( \Arond^2 \right) \right] \Arond^2 - \frac{1}{6} \left[ \left( \tr_g \Arond \right)^3 - 3 \tr_g\left( \Arond^2 \right) \left( \tr_g \Arond \right) +2 \tr_g\left( \Arond^3 \right) \right] \Arond + \det\left( \Arond \right) g = 0.
		\end{equation*}
		Since $\tr_g\Arond = 0$, we have
		\begin{equation}\label{cayley_hamilton}
			\Arond^4 -\frac{1}{2} \tr_g \left( \Arond^2 \right) \Arond^2 - \frac{1}{3} \tr_g \left( \Arond^3 \right) \Arond + \det\left( \Arond \right) g = 0.
		\end{equation}
		Taking the trace of this relation, we obtain
		\begin{equation*}
			\tr_g\left( \Arond^4 \right) = \frac{1}{2} \left[ \tr_g \left( \Arond^2 \right) \right]^2 - 4\det\left( \Arond \right).
		\end{equation*}
		We also obtain an expression for the inverse:
		\begin{equation*}
			\Arond\left[ \Arond^3 - \frac{1}{2} \tr_g \left( \Arond^2 \right) \Arond - \frac{1}{3} \tr_g\left( \Arond^3 \right) g \right] = - \det\left( \Arond \right) g,
		\end{equation*}
		Hence, the inverse is given by
		\begin{equation*}
			\Arond^{-1} = \frac{-1}{\det\left( \Arond \right) } \left[ \Arond^3 - \frac{1}{2} \tr_g \left( \Arond^2 \right) \Arond - \frac{1}{3} \tr_g\left( \Arond^3 \right) g \right].
		\end{equation*}
		We compute the trace of $\Arond^{-1}$:
		\begin{equation*}
			\tr_g \Arond^{-1} = \frac{-1}{\det\left( \Arond \right) } \left[ \tr_g \Arond^3 - \frac{4}{3} \tr_g \Arond^3 \right] = \frac{1}{3\det\left( \Arond \right) } \tr_g\left( \Arond^3\right).
		\end{equation*}
		We compute the norm of $\Arond^{-1}$:
		\begin{align}
			\left| \Arond^{-1} \right|^2_g =& \det\left( \Arond \right)^{-2} \Bigg[ \left| \Arond^3 \right|^2_g + \frac{1}{4} \left[ \tr_g \Arond^2 \right]^2 \left| \Arond \right|^2_g + \frac{1}{9} \left[ \tr_g \Arond^3 \right]^2 |g|^2_g \label{eq:norm_Ainv1}\\[2mm]
			&- \tr_g\left(\Arond^2 \right) \scal{\Arond^3}{\Arond}_g - \frac{2}{3} \tr_g\left(\Arond^3 \right) \scal{\Arond^3}{g}_g + \frac{1}{3} \tr_g\left( \Arond^2\right)\tr_g\left(\Arond^3\right) \scal{\Arond}{g}_g \Bigg] \label{eq:norm_Ainv2}
		\end{align}
		For \eqref{eq:norm_Ainv2}, we have
		\begin{align*}
			\scal{\Arond}{g}_g &= \tr_g\Arond =0, \\
			\scal{\Arond^3}{g}_g &= \tr_g \left(\Arond^3 \right), \\
			\scal{\Arond^3}{\Arond}_g &= \tr_g\left(\Arond^4 \right) = \frac{1}{2} \left[ \tr_g \left( \Arond^2 \right) \right]^2 - 4\det\left( \Arond \right).
		\end{align*}
		For \eqref{eq:norm_Ainv1}, we have $|g|^2_g = \tr_g(g) = 4$ and 
		\begin{align*}
			\left| \Arond^3 \right|^2_g &= \tr_g\left(\Arond^6 \right) = \tr_g\left[ \Arond^2 \left( \frac{1}{2} \tr_g \left(\Arond^2 \right) \Arond^2 + \frac{1}{3} \tr_g\left( \Arond^3 \right) \Arond - \det\left(\Arond\right)g \right) \right] \\[2mm]
			&= \frac{1}{2}\tr_g\left(\Arond^2 \right) \tr_g\left(\Arond^4 \right) + \frac{1}{3} \tr_g\left(\Arond^3 \right)^2-\det\left(\Arond \right) \tr_g\left(\Arond^2 \right) \\[2mm]
			&= \frac{1}{4} \left[ \tr_g \Arond^2 \right]^3 - 3\left( \det \Arond \right) \tr_g \left(\Arond^2 \right) + \frac{1}{3} \left(\tr_g \Arond^3 \right)^2.
		\end{align*}
		Thus, \eqref{eq:norm_Ainv1}--\eqref{eq:norm_Ainv2} reduces to
		\begin{equation*}
			\begin{aligned} 
				\left| \Arond^{-1} \right|^2_g &= \det\left( \Arond \right)^{-2} \Bigg[ \frac{1}{4} \left( \tr_g \Arond^2 \right)^3 - 3 \left(\det\Arond \right) \left( \tr_g\Arond^2 \right) + \frac{1}{3} \left( \tr_g \Arond^3 \right)^2 + \frac{1}{4} \left( \tr_g \Arond^2 \right)^3 + \frac{4}{9} \left( \tr_g \Arond^3 \right)^2 \\[2mm]
				&\quad - \frac{1}{2} \left( \tr_g \Arond^2 \right)^3 + 4\left( \tr_g \Arond^2 \right) \left( \det\Arond \right) - \frac{2}{3} \left(\tr_g \Arond^3 \right)^2 \Bigg] \\[2mm]
				&= \det\left( \Arond \right)^{-2} \Bigg[ \left( \tr_g \Arond^2 \right) \left(\det \Arond \right) + \frac{1}{9} \left( \tr_g \Arond^3 \right)^2 \Bigg].
			\end{aligned}
		\end{equation*}
	\end{proof}
	
	By \cite[Chapter 4, Section 9]{guillemin2010}, we have the following Gauss--Bonnet formula in dimension 4:
	\begin{equation}\label{gauss_bonnet}
		\int_\Sigma \det A\ \dd\vol_g = \frac{4\pi^2}{3} \chi(\Sigma).
	\end{equation}
	We now express $\det A$ in terms of $H$ and $\Arond$.
	
	\begin{lemma}
		Let $\Phi \colon \Sigma^4\to\R^5$ be a smooth immersion. Then
		\begin{equation}\label{det_A}
			\det A = H^4 - \frac{1}{2}H^2 |\Arond|^2_g + \frac{1}{3} H \left( \tr_g\Arond^3 \right) + \frac{1}{8} |\Arond|^4_g - \frac{1}{4} \left(\tr_g \Arond^4 \right).
		\end{equation}
	\end{lemma}
	
	\begin{proof}
		We apply the Cayley--Hamilton theorem to the matrix $A g^{-1}$:
		\begin{align*}
			0 & = A^4 - \left( \tr_g A \right) A^3 + \frac{1}{2}\left[ \left( \tr_g A \right)^2 - \tr_g \left( A^2 \right) \right] A^2 - \frac{1}{6} \left[ \left( \tr_g A \right)^3 - 3 \tr_g\left( A^2 \right) \left( \tr_g A \right) +2 \tr_g\left( A^3 \right) \right] A + \det\left( A \right) g \\[2mm]
			& = A^4 - 4H A^3 + \frac{1}{2}\left[ 16H^2 - |A|^2_g \right] A^2 - \frac{1}{6} \left[ 64H^3 - 12 |A|^2_g H +2 \tr_g\left( A^3 \right) \right] A + \det\left( A \right) g .
		\end{align*}
		By tracing this relation, we obtain
		\begin{align*}
			0 =& \tr_g(A^4) - 4H\tr_g(A^3) +(6H^2 - \frac{1}{2}|\Arond|^2_g) |A|^2_g - \frac{4}{3}(32 H^3 - 6|A|^2_g H + \tr_g(A^3)) H + 4\det(A)\\[2mm]
			=& \tr_g\left[ (\Arond + Hg)^4 \right] -4H \tr_g\left[(\Arond +Hg)^3 \right] +(6H^2 - \frac{1}{2}|\Arond|^2_g) (4H^2 + |\Arond|^2_g) \\[2mm]
			& - \frac{4}{3}\left(32 H^3 - 6(4H^2 + |\Arond|^2_g) H + \tr_g\left[(\Arond +Hg)^3 \right] \right) H +4\det A \\[2mm]
			=& \tr_g\left[ \Arond^4 +4 H\Arond^3 +6H^2 \Arond^2 + 4H^3 \Arond + H^4 g \right]  - 4H \tr_g\left[ \Arond^3 + 3H \Arond^2 + 3H^2 \Arond + H^3g \right]\\[2mm]
			&+ 24H^4 + 4H^2 |\Arond|^2_g - \frac{1}{2}|\Arond|^4_g \\[2mm]
			&-\frac{4}{3} \left( 8H^3 - 6|\Arond|^2_g H + \tr_g\left[ \Arond^3 + 3H \Arond^2 + 3H^2 \Arond + H^3g \right] \right) H \\[2mm]
			&+ 4\det A.
		\end{align*}
		Therefore, we obtain
		\begin{equation*}
			0 = 4\det A + \tr_g(\Arond^4) - \frac{1}{2}|\Arond|^4_g - \frac{4}{3} H \tr_g(\Arond^3) + 2H^2 |\Arond|^2_g -4H^4.\qedhere
		\end{equation*}
	\end{proof}
	
	\section{Curvature identities for Euclidean hypersurfaces}\label{sec:curvature_Phi}
	
	Consider a smooth immersion $\Phi \colon \Sigma^4 \to \R^5$, where $\Sigma$ is a closed $4$-manifold. Let $\xi$ be the flat metric on $\R^5$ and let $g = \Phi^*\xi$. In \cref{subsec:Simons}, we prove Simons' identity. In \cref{subsec:Curvature}, we express all the curvature tensors of $\Phi(\Sigma)$ in terms of its second fundamental form $A$ or its traceless part $\Arond$, with the goal of proving Lemma \ref{lm:integral_gH}.

	\subsection{Simons' identity}\label{subsec:Simons}
	By \cite[Theorem 4]{petersen2006}, we have the following formula.
	\begin{theorem}
		The Gauss--Codazzi equations can be written as follows. If $x,y,z$ are tangent vector fields to $\Phi(\Sigma)\subset \R^5$, then
		\begin{equation*}
			\Riem^{\R^5}(x,y,z,N) = \g_y A(x,z) - \g_x A(y,z)
		\end{equation*}
	\end{theorem}
	
	\begin{corollary}
		The Codazzi equations can be written as
		\begin{equation}\label{gauss_codazzi_Phi}
			\forall i,j,k,\qquad \g^g_i A_{jk} = \g^g_j A_{ik}.
		\end{equation}
	\end{corollary}
	
	\begin{proof}
		The Riemann tensor of $\R^5$ is $0$.
	\end{proof}
	
	\begin{proposition}
		Simons' identity can be written as
		\begin{equation}\label{simons_identity_Phi}
			\lap_g A_{ij} = 4\g^g_i \g^g_j H + 4 H \left( A^2 \right)_{ij} - |A|^2_g A_{ij}.
		\end{equation}
	\end{proposition}

	\begin{proof}
		By the Codazzi equations, we have
		\begin{equation*}
			\lap_g A_{ij} = \tr_g\left[ \left( \g^g \right)^2 A_{ij} \right] = g^{kl} \g^g_k \g^g_l A_{ij} = g^{kl} \g^g_k \g^g_i A_{lj}.
		\end{equation*}
		We commute the second covariant derivatives of a $(0,2)$-tensor on $\Phi(\Sigma)$:
		\begin{align*}
			\lap_g A_{ij} &= g^{kl} \left[ \g^g_i \g^g_k A_{lj} + \left( \Riem^{\Phi(\Sigma)}\right)_{kil}^{\ \ \ \alpha} A_{\alpha j} + \left( \Riem^{\Phi(\Sigma)}\right)_{kij}^{\ \ \ \alpha} A_{\alpha l} \right] \\[2mm]
			& =  g^{kl} \g^g_i \g^g_j A_{lk} + g^{kl}\left[ A_{kl} A_i^{\ \alpha} -A_k^{\ \alpha} A_{il} \right] A_{\alpha j} + g^{kl}\left[ A_{kj} A_i^{\ \alpha} -A_k^{\ \alpha} A_{ij} \right] A_{\alpha l} .
		\end{align*}
		Hence,
		\begin{equation*}
			\lap_g A_{ij} = 4\g^g_i \g^g_j H + 4 H \left( A^2 \right)_{ij} - \left( A^3 \right)_{ij}  + \left(A^3 \right)_{ij}- |A|^2_g A_{ij}.\qedhere
		\end{equation*}
	\end{proof}
	
	\subsection{Curvature tensors}\label{subsec:Curvature}
	
	The following result can be found in \cite[Theorem II.2.1]{chavel2006}:
	\begin{theorem}
		For any vector fields $x,y,z,t$ along $\Phi(\Sigma)\subset (\R^5,\xi)$, the Gauss equation can be written as
		\begin{equation*}
			\Riem^g(x,y,z,t) = \scal{A(x,z)}{A(y,t)}_\xi - \scal{A(x,t)}{ A(y,z)}_\xi
		\end{equation*}
	\end{theorem}
	
	\begin{lemma}\label{riemann_Phi}
		In local coordinates, we have
		\begin{align*}
			\Riem^g_{ijkl} =&  A_{ik}A_{jl} - A_{il} A_{jk} \\[2mm]
			=& \Arond_{ik}\Arond_{jl} - \Arond_{il}\Arond_{jk} + H\Big( \Arond_{ik} g_{jl} + g_{ik}\Arond_{jl} - \Arond_{il}g_{jk} - \Arond_{jk}g_{il}\Big) + H^2 \Big( g_{ik} g_{jl} - g_{il} g_{jk} \Big).
		\end{align*}
	\end{lemma}
	\begin{proof}
		Using $A = \Arond +Hg$, we obtain
		\begin{align*}
			\Riem^g_{ijkl} =&(\Arond_{ik} + Hg_{ik} )(\Arond_{jl} + Hg_{jl}) -(\Arond_{il} + Hg_{il})(\Arond_{jk} + Hg_{jk}) \\[2mm]
			=& \Arond_{ik}\Arond_{jl} - \Arond_{il}\Arond_{jk} +H\Big( \Arond_{ik} g_{jl} + g_{ik}\Arond_{jl} - \Arond_{il}g_{jk} - \Arond_{jk}g_{il}\Big) + H^2 \Big( g_{ik} g_{jl} - g_{il} g_{jk} \Big).
		\end{align*}
	\end{proof}

	\begin{lemma}\label{ricci_Phi}
		In local coordinates, we have
		\begin{equation*}
			\Ric^g_{\alpha \beta} = 4H A_{\alpha \beta} - A_\alpha^{\ \gamma} A_{\gamma\beta} = -\Arond^2_{\alpha\beta} + 2H\Arond_{\alpha\beta} + 3H^2g_{\alpha\beta}.
		\end{equation*}
	\end{lemma}
	
	\begin{proof}
		We take the trace of \eqref{riemann_Phi}:
		\begin{equation*}
			\Ric^g_{ik} = g^{jl} \Riem^g_{ijkl} = g^{jl}\Big( A_{ik}A_{jl} - A_{il} A_{jk} \Big) = 4H A_{ik} - A_i^{\ j} A_{jk}.
		\end{equation*}
		In terms of $\Arond$, we obtain
		\begin{equation*}
			\Ric^g_{ik} = 4H\Arond_{ik} + 4H^2g_{ik} - \Arond^2_{ik} -2H\Arond_{ik} - H^2g_{ik} = 2H\Arond_{ik} + 3H^2 g_{ik} - \Arond^2_{ik}.\qedhere
		\end{equation*}
	\end{proof}

	\begin{lemma}\label{scal_Phi}
		We have
		\begin{equation*}
			\Scal^g = 16 H^2 - |A|^2_g = 12H^2 - |\Arond|^2_g.
		\end{equation*}
	\end{lemma}
	
	\begin{proof}
		We take the trace of Lemma \ref{ricci_Phi}:
		\begin{equation*}
			\Scal^g = g^{\alpha\beta} \Ric^{\Phi(\Sigma)}_{\alpha\beta} = g^{\alpha\beta} \Big( 4H A_{\alpha \beta} - A_\alpha^{\ \gamma} A_{\gamma\beta} \Big)  = 16 H^2 - |A|^2_g.\qedhere
		\end{equation*}
	\end{proof}

	\begin{lemma}\label{lemma_laplace_scalar}
		We have
		\begin{align*}
			\lap_g \Scal^g =& -8\g^g_i \g^g_j \Big(A^{ij} H - 4H^2 g^{ij} \Big) +32 |\g^g H|^2_g  -2 |\g^g A|^2_g  -8 H\tr_g(A^3) +2 |A|^4_g \\[2mm]
			=& -8\g^g_i \g^g_j \Big(A^{ij} H - 4H^2 g^{ij} \Big) +32 |\g^g H|^2_g  -2 |\g^g A|^2_g -8 H^2 |\Arond |^2_g -8 H\tr_g\Arond^3 +2 |\Arond|^4_g .
		\end{align*}
	\end{lemma}
	
	\begin{proof}
		We contract Simons' identity \eqref{simons_identity_Phi} with $A^{ij}$:
		\begin{align}
			A^{ij} \g_i^g \g^g_j H &= \frac{1}{4} \Big( A^{ij} \lap_g A_{ij}  - 4H \tr_g(A^3)  + |A|^4_g \Big) \nonumber \\[2mm]
			&= \frac{1}{8} \lap_g (|A|^2_g) - \frac{1}{4} |\g^g A|^2_g - H\tr_g(A^3) + \frac{1}{4} |A|^4_g. \label{eq:lap_Scal1}
		\end{align}
		On the other hand, using the Gauss--Codazzi equations, we have
		\begin{align*}
			A^{ij} \g_i^g \g^g_j H &= \g^g_i\Big( A^{ij} \g^g_j H \Big) - \Big(\g^g_j A^{ij} \Big)\Big( \g^g_i H \Big) \\[2mm]
			&= \g^g_i \g^g_j \Big(A^{ij} H \Big) - \g^g_i \Big( (\g^g_j A^{ij}) H \Big)- \Big(\g^g_j A^{ij} \Big)\Big( \g^g_i H \Big) \\[2mm]
			&= \g^g_i \g^g_j \Big(A^{ij} H \Big) - \g^g_i \Big( ((\g^g)^i A_j^{\ j}) H \Big)- \Big((\g^g)^i A_j^{\ j} \Big)\Big( \g^g_i H \Big).
		\end{align*}
		Therefore, we have
		\begin{align*}
			A^{ij} \g_i^g \g^g_j H &= \g^g_i \g^g_j \Big(A^{ij} H \Big) - 4\Bigg[ \g^g_i\Big( ((\g^g)^i H) H \Big) + \Big(  (\g^g)^i H \Big) \Big( \g^g_i H \Big) \Bigg] \\[2mm]
			&= \g^g_i \g^g_j \Big(A^{ij} H \Big) - 4\Bigg[ \frac{1}{2} \g^g_i (\g^g)^i H^2 + |\g^g H|^2_g \Bigg] \\[2mm]
			&= \g^g_i \g^g_j \Big(A^{ij} H \Big) - 2 \lap_g H^2 - 4 |\g^g H|^2_g \\[2mm]
			&= \g^g_i \g^g_j \Big(A^{ij} H - 4H^2 g^{ij} \Big) + 2 \lap_g H^2 - 4 |\g^g H|^2_g .
		\end{align*}
		Together with \eqref{eq:lap_Scal1}, we obtain
		\begin{equation*}
			\frac{1}{8} \lap_g (|A|^2_g) - \frac{1}{4} |\g^g A|^2_g - H\tr_g(A^3) + \frac{1}{4} |A|^4_g = \g^g_i \g^g_j \Big(A^{ij} H - 4H^2 g^{ij} \Big) + 2 \lap_g H^2 - 4 |\g^g H|^2_g.
		\end{equation*}
		We isolate the term involving $\lap_g(|A|^2)$:
		\begin{equation*}
			\frac{1}{8} \lap_g (|A|^2_g) = \g^g_i \g^g_j \Big(A^{ij} H - 4H^2 g^{ij} \Big) + 2 \lap_g H^2 - 4 |\g^g H|^2_g  + \frac{1}{4} |\g^g A|^2_g + H\tr_g(A^3) - \frac{1}{4} |A|^4_g.
		\end{equation*}
		Using \eqref{scal_Phi}, we obtain
		\begin{align}
			\lap_g \Scal^g =& 16 \lap_g (H^2) - \lap_g |A|^2_g \nonumber \\[2mm]
			=& 16 \lap_g (H^2) -8\g^g_i \g^g_j \Big(A^{ij} H - 4H^2 g^{ij} \Big) -16 \lap_g H^2 +32 |\g^g H|^2_g  -2 |\g^g A|^2_g -8 H\tr_g(A^3) +2 |A|^4_g \nonumber \\[2mm]
			=& -8\g^g_i \g^g_j \Big(A^{ij} H - 4H^2 g^{ij} \Big) +32 |\g^g H|^2_g  -2 |\g^g A|^2_g  -8 H\tr_g(A^3) +2 |A|^4_g \label{eq:lap_Scal22}
		\end{align}
		We express $|A|^4_g$ in terms of $\Arond$:
		\begin{equation*}
			|A|^4_g = \left( |\Arond|^2_g + 4H^2 \right)^2 = |\Arond|^4_g + 8H^2 |\Arond|^2_g+16H^4.
		\end{equation*}
		We express $\tr_g A^3$ in terms of $\Arond$:
		\begin{equation*}
			\tr_g A^3 = \tr_g \left( \Arond^3 + 3H \Arond^2 + 3 H^2 \Arond + H^3 g \right) = \tr_g\Arond^3 +3H|\Arond|^2_g +4H^3.
		\end{equation*}
		Substituting these two expressions into \eqref{eq:lap_Scal22}, we obtain
		\begin{equation*}
			\lap_g \Scal^g =-8\g^g_i \g^g_j \Big(A^{ij} H - 4H^2 g^{ij} \Big) +32 |\g^g H|^2_g  -2 |\g^g A|^2_g -8 H\left(\tr_g\Arond^3 +3H|\Arond|^2_g +4H^3 \right) +2 \left( |\Arond|^4_g + 8H^2 |\Arond|^2_g+16H^4 \right).
		\end{equation*}
		Hence, it holds
		\begin{equation*}
			\lap_g \Scal^g = -8\g^g_i \g^g_j \Big(A^{ij} H - 4H^2 g^{ij} \Big) +32 |\g^g H|^2_g  -2 |\g^g A|^2_g -8 H\tr_g\Arond^3 -8H^2|\Arond|^2_g  +2 |\Arond|^4_g.\qedhere
		\end{equation*}
	\end{proof}
	
	Integrating $\lap_g \Scal^g$, we obtain the following lemma.
	\begin{lemma}\label{lm:integral_gH}
		We have
		\begin{align*}
			\int_\Sigma |\g^g H|^2_g\, \dd\vol_g &= \frac{1}{12} \int_\Sigma \left( |\g^g \Arond|^2_g + 4H^2 |\Arond|^2_g + 4H\tr_g(\Arond^3) -  |\Arond|^4_g\right) \dd\vol_g \\[2mm]
			&= \frac{1}{16} \int_\Sigma \left( |\g^g A|^2_g + 4H^2 |\Arond|^2_g + 4H\tr_g(\Arond^3) -  |\Arond|^4_g\right) \dd\vol_g.
		\end{align*}
	\end{lemma}
	
	\begin{proof}
		By Lemma \ref{lemma_laplace_scalar}, we have
		\begin{equation*}
			0 = \int_\Sigma \left( 32|\g^g H|^2_g - 2|\g^g A|^2_g - 8H^2 |\Arond|^2_g -8H\tr_g(\Arond^3) + 2|\Arond|^4_g\right) \dd\vol_g.
		\end{equation*}
		Using $|\g^g A|^2_g = 4|\g^g H|^2_g + |\g \Arond|^2_g$, we obtain the result.
	\end{proof} 
	
	\bibliography{biblio}	
	\bibliographystyle{alpha}
	
\end{document}